\documentclass[a4paper, 11pt]{amsart}
\usepackage[T1]{fontenc}
\usepackage[utf8]{inputenc}
\usepackage{amsmath}
\usepackage{amsthm}
\usepackage{amsfonts}
\usepackage{amssymb}
\usepackage{graphicx}
\usepackage[english]{babel}
\usepackage{version} 
\usepackage{nicefrac}
\usepackage{tipa} 
\usepackage{hyperref}
\usepackage{soul}
\usepackage{enumitem}
\usepackage{doc}
\usepackage[all,cmtip]{xy}
\usepackage{mathtools}
\usepackage{tikz-cd}

\theoremstyle{definition}
\newtheorem{defin}{Definition}[section]
\newtheorem{ex}[defin]{Example}
\newtheorem*{ex*}{Example}
\theoremstyle{plain}
\newtheorem{theo}[defin]{Theorem}
\newtheorem{lemma}[defin]{Lemma}
\newtheorem{obs}[defin]{Remark}
\newtheorem*{obs*}{Remark}

\newtheorem*{obss*}{Remarks}
\newtheorem{prop}[defin]{Proposition}
\newtheorem*{prop*}{Proposition}
\newtheorem{cor}[defin]{Corollary}

\newtheorem{theorem}{Theorem}
\newtheorem*{theorem*}{Theorem}
\newtheorem{corollary}[theorem]{Corollary}
\newtheorem*{corollary*}{Corollary}

\usepackage{xcolor}
\definecolor{light-gray}{gray}{0.1}

\makeatletter
\def\@setthanks{\vspace{-\baselineskip}\def\thanks##1{\@par##1\@addpunct.}\thankses}
\makeatother

\title[Convergence and collapsing of CAT$(0)$-lattices]
{ 
Convergence and collapsing  \\ of  CAT$(0)$-lattices}

\author{Nicola Cavallucci}
\thanks{N. Cavallucci has been partially supported by the SFB/TRR 191, funded by the DFG}
\author{Andrea Sambusetti}
\thanks{A. Sambusetti is member of GNSAGA and acknowledges the support of INdAM during the preparation of this work.}
 
\date{\today}

\begin{document}
	\maketitle

 \begin{abstract}
We study the theory of convergence for CAT$(0)$-lattices (that is groups $\Gamma$ acting geometrically on proper, geodesically complete CAT$(0)$-spaces)  and their quotients (CAT$(0)$-orbispaces). We describe some splitting and collapsing phenomena, explaining precisely how these action can degenerate to a possibly non-discrete limit action.
Finally, we prove a compactness theorem for the class of compact CAT$(0)$-homology orbifolds, and some applications: an isolation result for flat orbispaces and an entropy-pinching theorem.
\end{abstract}	

 \tableofcontents
	
\newpage	
	\section{Introduction}

The theory of CAT$(0)$-groups, i.e.
groups admitting a geometric action on some CAT$(0)$-space,	
whose roots can be traced back to the '70s with the works of Gromoll-Wolf \cite{GW71}
and Lawson-Yau \cite{LY72}
on fundamental groups of compact, nonpositively curved manifolds, has flourished in the last twenty years, with stunning developments  in geometric group theory and applications to low dimensional topology; let us just mention all the research opened by the rank rigidity conjecture for CAT$(0)$-spaces
and the methods coming from the theory of cubulable groups, with application to the virtually Haken conjecture.

 The works of Caprace and Monod \cite{CM09a}, \cite{CM09b} shed some new light on the structure of CAT$(0)$-spaces possessing a geometric action; as we will see, their results, together with the  deep, recent insights  on the topology of locally compact, geodesically  complete, CAT$(\kappa)$-spaces provided by  Lytchak and Nagano in \cite{LN19}, \cite{LN-finale-18}, will be of  great importance for our work.
\vspace{1mm}

In this paper, which should be thought as a companion paper of \cite{CS23}, and  pursues the study of packed CAT$(0)$-spaces and discrete group actions on such spaces initiated in  \cite{CavS20} and \cite{CS22},  we investigate the convergence of geometric  actions of groups on CAT$(0)$-spaces
along the line of the classical theory of convergence in Riemannian geometry. 
Namely, we will  be  interested in  {\em uniform}  CAT$(0)$-{\em lattices}, that is discrete,  cocompact isometry  groups $\Gamma$ of some CAT$(0)$-space $X$; one equivalently says  that $\Gamma$ acts {\em geometrically} on $X$. Moreover, throughout this paper  all CAT$(0)$-spaces  will be assumed to be  {\em proper}  and {\em geodesically complete}, which ensures many desirable geometric properties, such as the equality of topological dimension  and Hausdorff dimension, the existence of a canonical measure $\mu_X$, etc. (see Section \ref{sec-CAT} for fundamentals on CAT$(0)$-spaces). 
 Our goal is to  describe precisely the structure of  equivariant  Gromov-Hausdorff limits of sequences of (possibly collapsing) geometric actions on CAT$(0)$-spaces $\Gamma_n \curvearrowright X_n$, and of the corresponding quotients
$M_n =  \Gamma_n \backslash X_n$.
 As an application, we will present a compactness result (Corollary \ref{theo-intro-compactness}) and some rigidity and pinching theorems (Corollaries \ref{cor-flats}, \ref{cor-entropy}  below);  
see also some 
stability results in Section~\ref{sec-convergence} (Theorem \ref{prop-Riemannian-close-topological} and Corollaries  \ref{cor-homeo}, \ref{cor-euclfactor}).
\vspace{2mm}


 To begin with, let us denote by $$ \textup{CAT}_0 (D_0)$$
\vspace{-3mm}

\noindent  the class of uniform $\textup{CAT}(0)$-lattices\footnote{We emphasize that by {\em uniform   \textup{CAT}$(0)$-lattice} we mean   a CAT$(0)$-group $\Gamma$ with a fixed   faithful   geometric action of $\Gamma$ on some CAT$(0)$-space $X$, which explains the notation   $(\Gamma, X)$.}
 $(\Gamma, X)$ with   diam$(\Gamma \backslash X) \leq D_0$. We will also  say that  $\Gamma$ acts {\em $D_0$-cocompactly} on $X$ when diam$(\Gamma \backslash X) \leq D_0$, and that the space $X$ is {\em $D_0$-cocompact.}
   Notice that the constant $D_0$  is there simply to fix a scale, as the metric of any cocompact CAT$(0)$-space $X$ can be renormalized in order that it becomes $D_0$-cocompact.
 The lattice $\Gamma$ is called  {\em nonsingular} if  there exists at least a point $x \in X $ with trivial stabilizer: this is a mild assumption on the action which rules out  pathological  (although differently interesting) cases, see for instance \cite[Example 1.4]{CS22} and  \cite[Theorem 7.1]{BK90}. This condition is  automatically satisfied for instance when the   the lattice is torsion-free or $X$ is a homology manifold (see Section \ref{subsection-isometries} and \cite{CS22}).
 We  call  the quotient metric space  $M = \Gamma \backslash X$  a  {\em \textup{CAT}$(0)$-orbispace} 
and we will   say that the CAT$(0)$-orbispace  $M$  is {\em nonsingular}  if $\Gamma$ acts nonsingularly on $X$.
  Notice that  if $\Gamma$ is torsion-free then  $M$   is  a locally CAT$(0)$-space. 
\vspace{1mm}


The starting point of our study
is to understand when a family of uniform CAT$(0)$-lattices admits a limit (possibly not a lattice), i.e. the precompactness of our class $ \textup{CAT}_0 (D_0)$.
The noton of convergence that we will use is  the {\em equivariant  Gromov-Hausdorff  convergence} ({\em \`a la Fukaya}, cp. \cite{Fuk86}). 
We refer to Section \ref{sec-convergence} for the precise definition;  let us just mention here that, denoting by $B_{\Gamma}(x, r)$  the subset of elements $\gamma \in \Gamma$ moving $x \in X$ less than $r$,  saying that a sequence of lattices $(\Gamma_j,  X_j)$ converges towards a  limit action  $(\Gamma_\infty, X_\infty)$   simply means that there exist  Gromov-Hausdorff $\varepsilon$-approximations  $f_{\varepsilon}: B_{X_j}(x_j, \frac{1}{\varepsilon}) \rightarrow B_{X_\infty}(x_\infty,\frac{1}{\varepsilon}) $  between larger and larger balls of $X_j, X_\infty$  centered at basepoints $x_j$ and $x_\infty$,  which are  $\varepsilon$-equivariant with respect to maps $\phi_{\varepsilon}: B_{\Gamma_j}(x_j, \frac{1}{\varepsilon}) \rightarrow B_{\Gamma_\infty}(x_\infty,\frac{1}{\varepsilon}) $ (that is, with an equivariancy error smaller than $\varepsilon$), for $\varepsilon \rightarrow 0$.  
   For sequences of lattices which are  $D_0$-cocompact, the limit does not depend on the basepoints $x_j, x_\infty$; moreover, this is equivalent   (up to subsequences)  to taking the ultralimit group $\Gamma_\omega$ acting on the ultralimit metric space $X_\omega$ defined using any non-principal ultrafilter $\omega$, provided that  the  space $X_\omega$ is proper, see \cite{Cav21ter} and \S\ref{sub-GH}.

\noindent Now, the answer  to the existence  of   limits for a family of CAT$(0)$-lattices  is simple and comes from the very definition of equivariant GH-convergence: {\em any  family of
isometry groups $\Gamma_j < \textup{Isom}(X_j)$     (sub-)converges to some limit isometry group $\Gamma_\infty$ of a space $X_\infty$,   as soon as the spaces $X_j$ converge to $X_\infty$ with respect to the pointed Gromov-Hausdorff distance}, see 	Proposition \ref{prop-GH-ultralimit}). \linebreak 
Hence, the precompactness of a sequence in $\textup{CAT}_0 (D_0)$ reduces to  precompactness of the underlying spaces.   
\vspace{1mm}

Recall that a metric space $X$ is said to satisfy the {\em $P_0$-packing condition at scale $r_0$} (for short, $X$ is {\em $(P_0,r_0)$-packed}) if every ball of radius $3r_0$ in $X$ contains at most $P_0$ points that are $2r_0$-separated. 
For geodesically complete CAT$(0)$-spaces, the convexity of the distance function implies that a packing condition at some scale $r_0$ yields an explicit, uniform control of the packing function $\text{Pack}(R,r)$ of $X$ (that is, the maximum number of disjoint $r$-balls that one can pack in any $R$-ball), see \cite[Theorem 4.2]{CavS20}.
  This is exactly  Gromov's classical condition for precompactness of a family of spaces (also known as ``uniform compactness of $r$-balls'', cp. \cite{Gr81}, or ``geometrical boundedness'' \cite{DY}). 
\noindent Moreover, it is easy to see that any compact metric space,  as well as any metric space $X$ admitting an uniform lattice, is $(P_0, r_0)$-packed for some constants $P_0$ and $r_0$ (cp. the proof of \cite[Lemma 5.4]{Cav21ter}). 
Summarizing: the class $ \textup{CAT}_0 (D_0)$ has a natural filtration  
 \vspace{-4mm}
 
\begin{equation}
\label{eq:filtration}
\textup{CAT}_0(D_0) = \bigcup_{P_0, r_0} 
\textup{CAT}_0(P_0,r_0,D_0)
\end{equation}

 \noindent where  $\textup{CAT}_0(P_0,r_0,D_0)$ is the subset of $\textup{CAT}_0(D_0) $ made of lattices $(\Gamma, X)$ such that $X$ is 
$(P_0, r_0)$-packed,  and from the packing assumption one can deduce the following:

\begin{prop*}[Proposition \ref{lemma-GH-compactness-packing}]
	\label{lemma-GH-compactness-packing}
	A subset $  \mathcal{F}   \subseteq \textup{CAT}_0(D_0)$ is precompact  with respect to the equivariant pointed Gromov-Hausdorff convergence  if and only if there exist $P_0,r_0 > 0$ such that  $\mathcal{F}  \subseteq \textup{CAT}_0(P_0,r_0,D_0)$.
\end{prop*}

\noindent   We stress the “only if” part in the above statement: for a family of CAT$(0)$-lattices with uniformily bounded codiameter,   the uniform packing assumption at some fixed scale is a necessary condition in order to converge.  In view of this,  this assumption   is crucial   for studying equivariant Gromov-Hausdorff limits and is, in a sense,  minimal. \\
This $(P_0, r_0)$-packing condition should be thought as a weak, local  substitute of a  lower bound on the curvature; 
however, it  is much weaker than assuming the curvature bounded below in the sense of Alexandrov, or than a lower bound of the Ricci curvature in the  Riemannian case, and even of a CD$(\kappa,n)$ condition.
Indeed, the Bishop-Gromov's comparison theorem for Riemannian $n$-manifolds with Ric$_X \geq - (n-1) \kappa$ for $ \kappa \geq 0$ (or   its generalization to  CD$(\kappa,n)$ spaces, see for instance \cite{sturm}) yields a doubling condition  
 \begin{equation}\label{eq:doubling}\frac{\mu (B_X (x,2r)) }{\mu (B_X (x,r))} \leq  C(\kappa, n,r) \hspace{5mm} \end{equation}
from which the packing condition at scale $r_0 \leq r/4$ easily follows.
Notice however that  he class of CD$(\kappa,n)$ metric measure spaces which are CAT$(0)$ restricts to topological manifolds, by \cite{Kap-Ket-22}, whereas our class contains non-manifolds.
\noindent As proved in  \cite{CavS20},  for geodesically complete CAT$(0)$-spaces, the packing condition  at some scale is the same as a uniform upper bound of the canonical measure of all $r$-balls, a condition sometimes called {\em macroscopic scalar curvature bounded below}  cp. \cite{Gut10}, \cite{Sab20}. 
\vspace{2mm}

Coming to the problem of describing the possible limits of  lattices in our class, the first important distinction is 
  between collapsing and non-collapsing sequences. 
 We define the {\em free systole}
 of a lattice $\Gamma$ 
as
 $$\textup{sys}^\diamond(\Gamma,X) = \inf_{x\in X}\inf_{g\in \Gamma^\ast \setminus \Gamma^\diamond} d(x,gx)$$
where 
$\Gamma^\ast = \Gamma \setminus \{ \textup{id}\}$ and $\Gamma^\diamond$ is the subset of elliptic elements of $\Gamma$.\\
For nonsingular lattices, when assuming a bound on the diameter, the smallness of the free systole  is quantitatively equivalent to the smallness of the {\em diastole} of the lattice, that is  the invariant
$$\textup{dias}(\Gamma,X) = \sup_{x\in X}\inf_{g\in \Gamma^*} d(x,gx)$$
\noindent (see Proposition \ref{prop-vol-sys-dias}).
Accordingly, a lattice $(\Gamma , X )$ is said to be {\em $\varepsilon$-collapsed} if $\textup{sys}^\diamond (\Gamma,X) < \varepsilon$, and a sequence  $(\Gamma_j, X_j)_{j \in \mathbb{N}}$  is said to converge  {\em collapsing} to  $(\Gamma_\infty, X_\infty)$   
if $\limsup_{j\to+\infty} \textup{sys}^\diamond (\Gamma_j,X_j) = 0$; the sequence will be called {\em non-collapsing} otherwise. 
 
\noindent The collapsing condition for uniform $\textup{CAT}(0)$-lattices turns out to be  equivalent  to the the fact  that the dimension  of the limit orbispace  $M_\infty = \Gamma_\infty \backslash X_\infty$  is strictly smaller than the   dimension  of the orbispaces $M_j= \Gamma_j \backslash X_j$,   as we  will   prove   in Theorem \ref{theo-collapsing-characterization}, Section \ref{sub-characterization}; this is a very intuitive but nontrivial result, which  follows  a-posteriori from the analysis of both the collapsing and non-collapsing  case. 
Therefore, we can also say unambiguously that the orbispaces $M_j$ converge collapsing when $\textup{sys}^\diamond (\Gamma_j,X_j) \to 0$.
 
The following result shows that limits of  CAT$(0)$-lattices may well be non-discrete, both in the collapsing and in the non-collapsing case; they however always define a closed, {\em totally disconnected} group  of isometries of a canonical factor of the limit space  (which is another source of interest in the theory of totally disconnected groups):

\begin{theorem}[Limits of CAT$(0)$-lattices, Theorems \ref{theo-noncollapsed} \& \ref{theo-collapsed}]
\label{teor-intro-convergence}${}$\\
Let $(\Gamma_j,X_j)_{j \in \mathbb{N}} \subseteq \textup{CAT}_0(D_0)$  be a sequence of lattices converging in the equivariant Gromov-Hausdorff distance to a limit isometry group $(\Gamma_\infty, X_\infty)$, and let $M_j= \Gamma_j \backslash X_j$ the corresponding orbispaces. 
Then: 
\vspace{-1mm}

\begin{itemize}[leftmargin=4mm] 
\item[--] $X_\infty$ is a proper, geodesically complete  \textup{CAT}$(0)$-space,
\item[--] $\Gamma_\infty < \textup{Isom}(X_\infty)$ is a closed and $D_0$-cocompact group,
\item[--] the $M_j$'s converge to $M_\infty = \Gamma_\infty \backslash X_\infty$ 
in the Gromov-Hausdorff topology. 
\end{itemize} 

\noindent Moreover, if the sequence $(\Gamma_j,X_j)_{j \in \mathbb{N}}$ is:
\begin{itemize}[leftmargin=5mm] 
\item[i)]  {\em non-collapsing,}  then the group $\Gamma_\infty$ is totally disconnected;  in addition,  if each $\Gamma_j$ is nonsingular, then  $\Gamma_\infty$ is discrete, isomorphic to $\Gamma_j$ for $j\gg 0$, and   $M_\infty$ is equivariantly homotopic equivalent to $M_j$ for $j\gg 0$;
\item[ii)]  {\em collapsing,}  then $X_\infty$  splits isometrically as $X_\infty' \times \mathbb{R}^\ell$ with $\ell \geq 1$,   $\Gamma_\infty$ is not discrete  with identity component   $\Gamma_\infty^\circ \cong \mathbb{R}^\ell$,  acting as the group of translation of the  factor $\mathbb{R}^\ell$; finally  $ M_\infty = \Gamma_\infty'  \backslash X_\infty' $ where   $\Gamma_\infty' = \Gamma_\infty \slash \Gamma_\infty^\circ$ is closed and totally disconnected.
\end{itemize}
 \end{theorem}

\noindent If the limit $M_\infty$ is a Riemannian manifold, then the equivariant homotopy equivalence between $M_j$ and $M_\infty$ can be promoted to homeomorphism, as we will see in Section \ref{sec:riemannian} (Corollary \ref{prop-Riemannian-close-topological}).

 \begin{obss*}[Non-discrete limits] ${}$ \\
 {\em
\noindent i) A sequence of (singular) lattices can converge   to a non-discrete group  {\em even without collapsing}: an example  can be found in \cite{BK90}.  
 In that paper the authors  build an infinite, ascending chain of lattices $\Gamma_j$ of a regular tree $T$ with bounded valency, with same compact quotient, containing torsion subgroups $G_j < \Gamma_j$ whose order tends to infinity.\\
The sequence $(\Gamma_j, T)$ is non-collapsing (as $T$ does not split) and the  limit group   $\Gamma_\infty$  is  a totally disconnected, non-discrete subgroup of Isom$(T)$ (since if it was discrete, there would be a bound on the order of its finite subgroups, by \cite[Corollary II.2.8]{BH09}, contradicting the fact that the limit  contains  arbitrarily large finite subgroups too).

\vspace{1mm}
\noindent ii)  In the collapsing case, even assuming that the lattices are nonsingular or torsion-free, one cannot improve in general the conclusion saying that the totally disconnected group $\Gamma_\infty'$ is discrete: see Example \ref{ex-comm-not-normal} and Remark \ref{rem-totdiscvsdicrete}.
}
\end{obss*}

By Theorem \ref{teor-intro-convergence}, the class CAT$_0(D_0)$ is not closed, because of the existence of either singular  or collapsing sequences, which make the limit group non-discrete (while it is not difficult to see that the conditions of being $D_0$-cocompact and   $(P_0, r_0)$-packed are stable under Gromov-Hausdorff limits). The problem of determining a natural extension  of CAT$_0(D_0)$ which forms a compact class will be considered by the first author in \cite{Cav23}. 

\noindent However, as we will see in a moment, the limit $M_\infty$ of the quotient orbispaces $M_j=\Gamma_j \backslash X_j$ might still be a reasonable CAT$(0)$-orbispace, even when the sequence is collapsing. 
For this, we need to understand more deeply how collapsing occurs in our class, and it is exactly the purpose of the following.

\noindent The main idea  to give $M_\infty$ a CAT$(0)$-orbispace structure is that, for a  collapsing sequence,  the spaces  $X_j$ split  as $Y_j \times {\mathbb R}^\ell$ and  we would like to prove a corresponding splitting of the groups $\Gamma_j=\Gamma_{Y_j} \times {\mathbb Z}^\ell$, with each factor preserving the product decomposition; then, we would   show  that collapsing only happens in the Euclidean factor via the action of  ${\mathbb Z}^\ell$, and disintegrate the action of the continuous part $\Gamma^\circ_\infty$ of the limit group $\Gamma_\infty$ on $X_\infty=Y_\infty \times {\mathbb R}^\ell$.\linebreak
 Unfortunately, this picture breaks down for general nonsingular CAT$(0)$-lattices: as opposite to the Riemannian case (where for any lattice of a non-positively curved Riemannian manifold  with Euclidean factor $ {\mathbb R}^\ell$  it is always possible to  virtually split a free abelian subgroup of  rank $\ell$),  we will see in \S\ref{sec-groupsplitting}, Example \ref{ex-comm-not-normal},  an example of  a collapsing sequence of lattices $(\Gamma_j, X_j)$ belonging to $ \textup{CAT}_0 (P_0,r_0,D_0)$, where  $X_j$ is the product of a tree $T$ with ${\mathbb R}^2$,  but $\Gamma_j$ does not split any abelian group of positive rank (not even virtually). 
 Actually,  the  splitting of a CAT$(0)$-space $X$ as $X' \times {\mathbb R}^\ell$ forces a lattice $\Gamma$ of $X$ to act preserving the components (i.e. every  $g \in\Gamma$ acts as  $(g', g'')$, where $g' \in \text{Isom}(X')$ and $g'' \in \text{Isom}(\mathbb{R}^\ell)$),
but,	in general, the projections of $\Gamma$ on $\text{Isom}(X')$ and $\text{Isom}(\mathbb{R}^n)$ are not discrete, as \cite[Example 1]{CM19} shows;  this is precisely the obstruction to virtually splitting the $\mathbb{Z}^\ell$-factor from $\Gamma$.

\noindent On the other hand, the group splitting always holds (virtually) for sufficiently  collapsed uniform lattices of    CAT$(0)$-spaces whose full isometry group is Lie:

 \begin{theorem}[Group splitting in the Lie  setting, extract from Thm. \ref{theo-group-splitting-Lie}] 
	\label{theo-intro-group-splitting-Lie} ${}$ \\
Given $P_0,r_0,D_0$, there exist 
  $ \sigma_0^\ast \!=  \sigma_0^\ast (P_0,r_0,D_0) \!>0$ 
  and $I_0 = I_0(P_0,r_0,D_0)$ such that the following holds. 
	Let $X$ be a proper, geodesically complete, $\textup{CAT}(0)$-space whose isometry group is a Lie group, and let $X = Y \times \mathbb{R}^n$ be the splitting of the maximal Euclidean factor of $X$: 
	then, every uniform lattice  $\Gamma$ of  $X$ 
  has a normal, finite index subgroup $\check \Gamma$  which splits as $\check \Gamma = \check \Gamma_Y \times \mathbb{Z}^n$  where $\check \Gamma_Y$ (resp. $\mathbb{Z}^n$) acts discretely on $Y$ (resp. $\mathbb{R}^n$).  \\
  If  $\Gamma$ is nonsingular then 	 $\check \Gamma$  is nonsingular too.\\
 Moreover, if $X$	is $(P_0,r_0)$-packed and $\Gamma$ is $D_0$-cocompact then 
$\textup{sys}^\diamond (\Gamma_Y,Y) \!\geq \!\sigma_0^\ast$ \linebreak and 	 $\check \Gamma$ can be chosen of index  $[\Gamma: \check \Gamma] \leq I_0$.
\end{theorem}

\noindent   In this situation, 
 $M=\Gamma \backslash X$ finitely covers (in the sense of orbispaces) the orbispace $\check M := \check \Gamma \backslash X$, which  splits as a metric
product $\check N \times \mathbb{T}^{n}$, where \linebreak  $\check N = \check \Gamma_Y \backslash Y$ and $\mathbb{T}^{n} =  \mathbb{Z}^n \backslash \mathbb{R}^n$ is a torus. We will  say in this case that the orbispace  $M$ {\em  splits virtually}.

 
\vspace{2mm}
The first part of Theorem 	\ref{theo-group-splitting-Lie} is a pretty natural consequence of Caprace and Monod's work on isometry groups of CAT$(0)$-spaces. On the other hand, the control of the index of  $\check \Gamma$ in $\Gamma$ and the bound of the diastole on the non-Euclidean factor follow from a finiteness theorem proved in \cite{CS23} and a subtle splitting result which will be recalled in Section \ref{subsec-splitting}, Theorem \ref{theo-splitting-weak}.\linebreak
Both estimates are crucial to understand Gromov-Hausdorff limit of the quotient spaces $M_j = \Gamma_j \backslash X_j$ when the actions collapse.
As a result, we obtain a complete description of how nonsingular lattices collapse, in the Lie setting:

\begin{theorem}[Collapsing in the nonsingular, Lie  setting]
\label{teor-intro-Lie-nonsingular} ${}$ \\
Assume that a sequence of nonsingular lattices   $(
\Gamma_j,X_j)_{j \in \mathbb{N}} \subseteq \textup{CAT}_0 (D_0)$
converge  collapsing to $(\Gamma_\infty, X_\infty)$,  and that the groups $\textup{Isom}(X_j)$ are Lie. \\
Let $X_j=Y_j \times \mathbb{R}^j$ be the canonical splitting  of the Euclidean factor of $X_j$, \linebreak
let $X_\infty  \!=  X_\infty'  \times \mathbb{R}^\ell $ be the 
splitting  given  by  Theorem \ref{teor-intro-convergence}, with  $ \ell =\textup{dim} (\Gamma_\infty^\circ) \geq 1$, \linebreak
and  let $k_\infty =\lim_j k_j \geq \ell$ be the dimension of the Euclidean factor of $X_\infty$. Then:
\begin{itemize}[leftmargin=4mm] 
\item[--] the lattice $\Gamma_j$ has a normal subgroup of finite index $ \check \Gamma_j$ splitting as $\check \Gamma_{Y_j}  \!\times \!{\mathbb Z}^{k_j}$;
\item[--]  the orbispace  
 $M_j \!=\Gamma_j \backslash X_j $ virtually  splits as   $\check{N}_j \! \times   \mathbb{T}^{k_j}_j$, 
 where  $\check{N}_j \!= \Gamma_{Y_j}  \backslash Y_j$   \linebreak  is a nonsingular \textup{CAT}$(0)$-orbispace and
$\mathbb{T}^{k_j} \!= {\mathbb Z}^{k_j} \backslash {\mathbb R}^{k_j}$ is a flat torus.
\end{itemize}	

\noindent  Moreover:
\begin{itemize}[leftmargin=8mm] 
			\item[(i)] the $\check{N}_j $'s converge non-collapsing to a nonsingular orbispace $\check{N}_\infty$; 
			\item[(ii)] the tori $\mathbb{T}^{k_j}$ converge to a flat torus $\mathbb{T}^{k_\infty'}$, with
			$k_\infty'= k_\infty - \ell $;
			\item[(iii)] the  $M_j$'s converge to the nonsingular orbispace $M_\infty \! = \Gamma_\infty' \backslash X_\infty'$, where  $\Gamma_\infty'  \!=  \Gamma_\infty /\Gamma_\infty^\circ $  is  discrete;   
			\item[(iv)]  $M_\infty$ is the quotient of $\check{N}_\infty \times \mathbb{T}^{k_\infty'}$ by a finite group of isometries $\Lambda_\infty$.
		\end{itemize}	
	Finally, if  
 $(\Gamma_\infty, X_\infty) \in 	\textup{CAT}_0 (P_0,r_0, D_0)$ and $(\check \Gamma_{Y_\infty}, Y_\infty)$ is the limit of the   $(\check \Gamma_{Y_j}, Y_j)$'s,
then $\textup{sys}^\diamond(\check \Gamma_{Y_\infty},Y_\infty)\ge\sigma_0^\ast (P_0, r_0, D_0)$  and   $|\Lambda_\infty| \leq I_0(P_0, r_0, D_0)$ \linebreak (here $\sigma_0^\ast$ and $I_0$ are the same as in Theorem \ref{theo-intro-group-splitting-Lie}).
\end{theorem}

\noindent In other words,  the theorem says that the collapsing occurs in this class only by possibly shrinking to a point a flat torus (virtual) fiber. This  result is new even for collapsing sequences of compact Riemannian manifolds $M$ with nonpositive sectional curvature $-\kappa \leq k(M) \leq 0$. Notice that also the fact that $k_\infty $ equals the dimension of the Euclidean factor of $X_\infty$ is not trivial and follows by Theorem \ref{teor-intro-convergence}, see Corollary \ref{cor-euclfactor}.
 \vspace{2mm}

 An interesting case where the above convergence theorem applies is the one of {\em \textup{CAT}$(0)$-homology orbifolds.} 
A metric space  $M$ is a \textup{CAT}$(0)$-homology orbifold if it can be realized as the quotient $M=\Gamma \backslash X$  of a proper CAT$(0)$-space $X$ which is a homology manifold (with respect to the induced topology) by a discrete group of isometries. We  denote respectively by 
	\vspace{-3mm}
	
	$${\mathcal H}{\mathcal O}\textup{-CAT}_0(D_0), \hspace{5mm}{\mathcal H}{\mathcal O}\textup{-CAT}_0(P_0, r_0, D_0)$$

\noindent the class	of compact CAT$(0)$-homology orbifolds $M \!=\Gamma \backslash X$ with diam$(M) \!\leq D_0$, \linebreak 
and the subclass of those 	such that $X$ is, moreover,  $(P_0,r_0)$-packed.\\
	Notice that, as a consequence of 
(\ref{eq:filtration}) we still have
	\vspace{-2mm}
	
	$${\mathcal H}{\mathcal O}\textup{-CAT}_0(D_0) = \bigcup_{P_0, r_0} {\mathcal H}{\mathcal O}\textup{-CAT}_0(P_0,r_0,D_0).$$


\noindent The class of  CAT$(0)$-homology orbifolds  is interesting  for many reasons. \linebreak
First remark that CAT$(0)$-homology manifolds are always geodesically complete (cp. \cite[Proposition II.5.12]{BH09}); also,  their lattices  always yield nonsingular orbispaces (see discussion in Section \ref{subsection-isometries}).  
	
\noindent More importantly, the class ${\mathcal H}{\mathcal O}\textup{-CAT}_0(P_0, r_0, D_0)$
 contains all nonpositively curved Riemannian manifolds with (uniformly) bounded sectional curvature. \\
	The advantage over   Riemannian manifolds is that being a homology manifold  (as well as the CAT$(0)$ and the $(P_0, r_0)$-packing conditions) is a property  which is stable under Gromov-Hausdorff convergence (see  \cite[Theorem 6.1]{CS22} and \cite[Lemma 3.3]{LN-finale-18}); while limits  of  of Riemannian manifolds, even without collapsing, yield at best  manifolds with $C^{1, \alpha}$-metrics, for which  the Riemannian curvature is not defined).  
	
\noindent Finally, the full isometry group  of a CAT$(0)$-homology manifold is  always a Lie group, as we will prove in Section \ref{subsection-isometries}, Proposition \ref{cor-homology-Lie},  which allows us to apply the results of Theorem \ref{teor-intro-Lie-nonsingular} in this class and show the following:

 \begin{corollary}[Gromov-Hausdorff  compactness]  
 \label{theo-intro-compactness}  ${}$\\
For all fixed $P_0,r_0$ and $D_0$, 	the class ${\mathcal H}{\mathcal O}\textup{-CAT}_0(P_0, r_0, D_0)$ is compact with respect to the Gromov-Hausdorff topology.  
	\end{corollary} 

\noindent Actually, among the most natural classes of metric spaces containing all Riemannian manifolds with bounded sectional curvature, ${\mathcal H}{\mathcal O}\textup{-CAT}_0(P_0, r_0, D_0)$ is  the smallest class  we know to be compact under Gromov-Hausdorff limits.



\noindent Comparing with the classical compactness theorems in Riemannian geometry (with bounded sectional curvature, e.g. \cite{Gro78b}, or bounded Ricci curvature  as in \cite{anderson90}, \cite{AC92}), 
 the novelty here is that, besides the wider generality of application (metric spaces, bounded packing instead of bounded curvature, torsion allowed), we do not assume any lower bound neither on the injectivity radius  nor on the volume, thus allowing convergence to spaces of smaller dimension.  \\
 We stress the fact that, in order to obtain a closed class, it is necessary to consider also quotients by discrete groups \emph{with torsion}. Indeed,  a sequence of compact, nonpositively curved manifolds with uniformly bounded curvature does not converge, generally, to a locally CAT$(0)$-manifold or even to a  locally  CAT$(0)$-homology manifold, as the following simple example shows.

\begin{ex*}   Let $M_j$ be the quotient of $\mathbb{R}^2$ by the discrete group $\Gamma_j$ generated by $\lbrace t, g_j \rbrace$, where $t$ is the translation of length $1$ along the $x$-axis and $g_j$ is the glide reflection with axis equal to the $y$-axis, with translation length $\frac{1}{j}$. \linebreak Each $M_j$ is a  smooth, flat $2$-manifold homeomorphic to the Klein bottle. \linebreak 
	However, the  $M_j$'s converge collapsing to a metric space $M_\infty$ which is the quotient of $\mathbb{R}^2$ by the limit $\Gamma_\infty$ of the groups $\langle t,g_j\rangle$. The limit space $M_\infty$ is homeomorphic to the closed interval $[0,1]$, which is not a homology manifold since it has boundary points; but it clearly is a (flat) Riemannian orbifold.
 \end{ex*}

	\begin{obs*} \emph{One can sharpen  Corollary \ref{theo-intro-compactness} in the Riemannian setting. 		Suppose that $M_j = \Gamma_j\backslash X_j \in \mathcal{HO}\textup{-CAT}_0(P_0,r_0,D_0)$  is a sequence of Riemannian orbifolds converging to $M_\infty$: then, $M_\infty$ is isometric to a quotient $\Gamma_\infty'\backslash X_\infty'$, where $X_\infty'$ is a \emph{topological} manifold. This follows by the same argument of the proof of Corollary \ref{theo-intro-compactness} and \cite[Theorem 1.3]{LN-finale-18}.}
	\end{obs*}

The following corollaries are two direct applications of the theory developed so far, and follow a classical scheme of proof:  we assume to have a sequence of counterexamples and pass to the limit, then we obtain a contradiction by using, respectively, a well-known result of Adams and Ballmann about virtually abelian lattices of CAT$(0)$-spaces, and a continuity result for the entropy (which we will prove in the Appendix).
	
\begin{corollary}[Isolation of $\mathbb{R}^n$ among cocompact CAT$(0)$-spaces] ${}$
\label{cor-flats} \\
There exists $\varepsilon = \varepsilon(D_0,n) > 0$ such that every  $D_0$-cocompact,  geodesically complete \textup{CAT}$(0)$-space $X$ satisfying $d_{\textup{pGH}}(X,\mathbb{R}^n) \!< \varepsilon$  is isometric to  $\mathbb{R}^n$.\linebreak
\end{corollary}

\noindent Notice that, in particular, this results implies that there do not exist arbitrarily small, $D_0$-periodic deformations of the Euclidean metric of $\mathbb{R}^n$ which remain  nonpositively curved, unless these deformations are flat.

Finally, recall that the  ({\em packing}, or {\em covering}) {\em entropy} of a proper CAT$(0)$-space $X$ is the asymptotic invariant defined as  
 $$\textup{Ent}(X) = \limsup_{R\to +\infty}\frac{1}{R}\log (\text{Pack}(R,r))$$
This can be equivalently defined as the critical exponent of any uniform lattice of $X$, and does not depend  on the small, chosen radius $r$, see   Section  \ref{subsection-packing-margulis}.\linebreak
It is well-known that  for every Hadamard manifold $X$ possessing a torsionless, uniform lattice $\Gamma$, Ent$(X)$ coincides with the topological entropy of the geodesic flow on the unit tangent bundle of the quotient manifold $M=\Gamma \backslash X$ \cite{Man79}.
Then, we have:

\begin{corollary}[Entropy-rigidity of $\mathbb{R}^n$] ${}$ 
\label{cor-entropy}  \\
There exists $h \!=\! h(P_0,r_0,D_0) \!> 0$ with the following property. \\
Let $X$ be a $D_0$-cocompact, geodesically complete  \textup{CAT}$(0)$-space which is $(P_0, r_0)$-packed: if \textup{Ent}$(X) \!< h$, then $X$ is isometric to  $\mathbb{R}^n$, for some $n$.
\end{corollary}


\noindent As a consequence, any non-flat,  compact  Riemannian manifold $M$ with sectional curvature $-\kappa \! \leq K_M  \! \leq 0$ and diam$(M) \! \leq D$ has topological entropy \linebreak
greater than  a positive, universal constant $h_0\!=h_0 (\kappa, D)$ (for $h_0\!= h(P_0,r_0,D) $, \linebreak where $(P_0,r_0)$ are the packing constant  deduced from the lower curvature bound given by (\ref{eq:doubling})).

	\vspace{2mm} 
 
	\small {\em
		\noindent {\sc Acknowledgments.} The authors thank P.E.Caprace, S.Gallot and A.Lytchak for many interesting discussions during the preparation of this paper, and D.Semola for pointing us to interesting references.}
	\normalsize

\section{Preliminaries on CAT$(0)$-spaces}
\label{sec-CAT}

We fix here some notation and recall some facts about CAT$(0)$-spaces. \\
Throughout the paper $X$ will be a {\em proper} metric space with distance $d$.\linebreak
The open (resp. closed) ball in $X$ of radius $r$, centered at $x$, will be denoted by $B_X(x,r)$ (resp.  $\overline{B}_X(x,r)$); we will often drop the subscript $X$ when the space  is clear from the context.\\
A {\em geodesic} in a metric space $X$ is an isometry $c\colon [a,b] \to X$, where $[a,b]$ is an interval of $\mathbb{R}$. The {\em endpoints} of the geodesic $c$ are the points $c(a)$ and $c(b)$; a geodesic with endpoints $x,y\in X$ is also denoted by $[x,y]$. A {\em geodesic ray} is an isometry $c\colon [0,+\infty) \to X$ and a geodesic line is an isometry $c\colon \mathbb{R} \to X$. A metric space $X$ is called   {\em geodesic}  if for every two points $x,y \in X$ there is a geodesic with endpoints   $x$ and $y$. 

\noindent A metric space $X$ is called CAT$(0)$ if it is geodesic and every geodesic triangle $\Delta(x,y,z)$  is thinner than its Euclidean comparison triangle   $\overline{\Delta} (\bar{x},\bar{y},\bar{z})$: that is,  for any couple of points $p\in [x,y]$ and $q\in [x,z]$ we have $d(p,q)\leq d(\bar{p},\bar{q})$ where $\bar{p},\bar{q}$ are the corresponding points in $\overline{\Delta} (\bar{x},\bar{y},\bar{z})$ (see   for instance \cite{BH09} for the basics of CAT(0)-geometry).
As a consequence, every CAT$(0)$-space is {\em uniquely geodesic}: for every two points $x,y$ there exists a unique geodesic with endpoints $x$ and $y$.

\noindent A CAT$(0)$-space $X$ is {\em geodesically complete} if every geodesic $c\colon [a,b] \to X$ can be extended to a geodesic line. For instance, if a CAT$(0)$-space is a homology manifold then it is always geodesically complete, see \cite[II.5.12]{BH09}.

\noindent The  {\em boundary at infinity} of a CAT$(0)$-space $X$ (that is, the set of equivalence classes of geodesic rays, modulo the relation of being asymptotic), endowed with the Tits distance,  will be denoted by $\partial X$, see \cite[Chapter II.9]{BH09}.

\noindent A subset $C$ of $X$ is said to be {\em convex} if for all $x,y\in C$ the geodesic $[x,y]$ is contained in $C$. Given a subset $Y\subseteq X$ we denote by $\text{Conv}(Y)$ the {\em convex closure} of $Y$, that is the smallest closed convex subset containing $Y$.\linebreak
If $C$ is a convex subset of a CAT$(0)$-space $X$ then it is itself CAT$(0)$, and its boundary at infinity $ \partial C$ naturally and isometrically embeds in $\partial X$.
\vspace{1mm}

  We will denote by HD$(X)$ and TD$(X)$ the Hausdorff and the topological dimension of a metric space $X$, respectively.  
 By \cite{LN19} we know that if $X$ is a  proper and geodesically complete CAT$(0)$-space then every point $x\in X$ has a well defined integer dimension in the following sense: there exists $n_x\in \mathbb{N}$ such that every small enough ball around $x$ has Hausdorff dimension equal to $n_x$. This defines a {\em stratification of $X$} into pieces of different integer dimensions: namely, if $X^k$ denotes the subset of points of $X$ with dimension $k$, then 
 \vspace{-5mm}
 
 $$X= \bigcup_{k\in \mathbb{N}} X^k.$$
The {\em dimension} of $X$ is the supremum of the dimensions of its points:  it coincides with the {\em topological dimension} of $X$, cp. \cite[Theorem 1.1]{LN19}.

\noindent Calling $\mathcal{H}^k$ the $k$-dimensional Hausdorff measure, the formula
 \vspace{-3mm}

$$\mu_X := \sum_{k\in \mathbb{N}} \mathcal{H}^k \llcorner X^k$$  defines a {\em canonical measure} on $X$ which is locally positive and locally finite.

\vspace{1mm}
 \subsection{Isometry groups and orbispaces}
\label{subsection-isometries}
${}$

\noindent The {\em translation length} of an isometry $g$ of a CAT$(0)$-space $X$
is defined as  $$\ell(g) := \inf_{x\in X}d(x,gx).$$ 
When the infimum is realized, the isometry $g$ is called {\em elliptic} if $\ell(g) = 0$ and {\em hyperbolic} otherwise. The {\em minimal set of $g$}, $\text{Min}(g)$, is defined as the subset of points of $X$ where $g$ realizes its translation length; notice that if $g$ is elliptic then $\text{Min}(g)$ is the subset of points fixed by $g$. An isometry is called \emph{semisimple} if it is either elliptic or hyperbolic; a subgroup $\Gamma$ of isometries of $X$
is  called \emph{semisimple} if all of its elements are semisimple.

\noindent Let $\text{Isom}(X)$ be the group of isometries of $X$, endowed with the compact-open topology: as $X$ is proper, it is a topological,  locally compact  group. \linebreak
A subgroup $\Gamma$ of isometries  is \emph{discrete} if it is  discrete as a subset of \textup{Isom}(X) (with respect to the compact-open topology). 
A {\em uniform lattice} 
of $X$ is a discrete isometry group $\Gamma$ such that the quotient metric space $\Gamma \backslash X$ is compact; then, we call diam$(\Gamma \backslash X)$ the  {\em codiameter} of $\Gamma$. We will often say  that   $\Gamma$  is $D$-cocompact if diam$(\Gamma \backslash X) \leq D$, and that the space  $X$ is {\em cocompact} if it admits a uniform lattice. Notice that a cocompact CAT$(0)$-space $X$ is always proper.

\noindent The action of $\Gamma$ on $X$ (and, by extension, the group $\Gamma$ itself) is called {\em nonsingular} if there exists a point $x_0 \in X$ such that the pointwise stabilizer $\text{Stab}_\Gamma(x_0)$ is trivial,  and {\em singular}  otherwise; see \cite[Example 1.4]{CS22} and \cite[Remark 5.3]{CS23} for examples of this phenomenon. 

\noindent Finally, when dealing with  isometry groups of CAT$(0)$-spaces with torsion, a   difficulty is that  there may exist nontrivial 
elliptic isometries which act as the identity on open sets. 
Following \cite[Chapter 11]{dlHG90}, a subgroup   $\Gamma$ of Isom$(X)$ will be called  {\em rigid} (or \emph{slim}, with the terminology used in \cite{CS22}) if for all $g \in \Gamma$ the subset $\text{Fix}(g)$ has empty interior. \\ 
Notice that a torsion-free group is trivially rigid, as well as any discrete group acting on a CAT$(0)$-homology manifold, as proved in \cite[Lemma 2.1]{CS22}). \linebreak
Also,  every rigid isometry group $\Gamma$ is automatically nonsingular since, by Baire's theorem, the union of fixed-point sets of all elliptic elements of $\Gamma$ has empty interior, cp. \cite[Proposition 2.11]{CS22}.
\vspace{1mm}

\noindent We will call the quotient metric space $M = \Gamma \backslash X$ of a proper, geodesically complete, CAT$(0)$-space $X$ by a discrete group of isometries $\Gamma$,   a  {\em \textup{CAT}$(0)$-orbispace}. We will also say that the CAT$(0)$-orbispace  $M$  is {\em nonsingular} (resp. {\em rigid}) if $\Gamma$ acts nonsingularly (resp. rigidly) on $X$. Notice that  if $\Gamma$ is torsion-free then  $M$   is  a locally CAT$(0)$-space.

\noindent When restricting our attention to groups $\Gamma$ acting {\em rigidly} on $X$ (i.e. such that every $g \in \Gamma$ acting as the identity on an open subset  is trivial) then our definition of CAT$(0)$-orbispace is equivalent to the notion of  {\em rigid, developable orbispace} as defined in \cite[Ch.11]{dlHG90} through an orbifold atlas, with CAT$(0)$ universal covering in the sense of orbispaces. 
However, our results will apply to all  CAT$(0)$-orbispaces  as defined above (with some distinction between the singular and nonsingular cases).
  
\begin{obs}\label{obs-samequotient}
{\em
  Notice that a  \textup{CAT}$(0)$-orbispace has more structure than simply the structure of a  metric space, including also the data of a group action on some  \textup{CAT}$(0)$-space. In this sense, it is worth to stress that (as opposite to the case of {\em Riemannian} orbifolds) two lattices $\Gamma_1, \Gamma_2$ of different  \textup{CAT}$(0)$-spaces $X_1, X_2$ may give the same metric quotient $M$.\\
  An  example of this (for CAT$(1)$-orbispaces) is the {\em $k$-triplex} $X$ constructed in \cite{nagano-triplex}: this is a proper, geodesically complete, purely $k$-dimensional CAT$(1)$-space (actually, a topological sphere) admitting a lattice $\Gamma \cong {\mathbb Z}_3$ such that $M= \Gamma \backslash X$ is isometric to the CAT$(1)$-orbispace $M'={\mathbb Z}_2 \backslash {\mathbb S}^k$, with ${\mathbb Z}_2$ acting as a rotation of $\pi$.\\
An example for  CAT$(0)$-orbispaces can be costructed by taking a ramified covering of a flat torus 
$T = {\mathbb Z}^2 \backslash {\mathbb R}^2$:
let $X$ be obtained by taking two copies $T_i \setminus \gamma_i$ of $T$ minus a geodesic segment $\gamma$, and then glueing together the resulting couples of boundaring edges $(\gamma_1^+,\gamma_1^-)$ with $(\gamma_2^-,\gamma_2^+)$. 
This yields a (metric) ramified covering $f: X \rightarrow T$ from a topological  surface of genus $2$ to the initial torus, with ramification locus given by the endpoints of $\gamma$. Notice that $X$ is a  metric space which has an  involutive isometry $\sigma$ such that $\langle\sigma\rangle \backslash X = T$, and it can be seen as the quotient of a CAT$(0)$-space $\tilde X$  by a lattice $\Gamma$ containing an elliptic element $\tilde \sigma$ of order $2$ (and  all its conjugates). Calling $\tilde \Gamma$ the subgroup of Isom$(\tilde X)$ generated by $\tilde{\sigma}$ and $\Gamma$, we have that $\tilde \Gamma \backslash \tilde X$ is also isometric to  $T$.
} 
\end{obs}  
 \newpage
  
On the other hand, the underlying metric space $M=\Gamma \backslash X$ determines $X$ and $\Gamma$  when $\Gamma$ acts freely  (since in this case $X$ is just the universal cover of $M$), or when $X$ is Riemannian, as the following result shows:

\begin{prop} Let $\Gamma$ and $\Gamma'$ be discrete isometry groups of two Hadamard manifolds $X$, $X'$ respectively.  
If $\Gamma \backslash X$ is isometric to $\Gamma' \backslash X'$ then there exists an isomorphism $\varphi \colon \Gamma \to \Gamma'$ and a $\varphi$-equivariant isometry  $F: X \rightarrow X'$.
\end{prop}

\begin{proof}
	 The space $M = \Gamma \backslash X = \Gamma' \backslash X$ admits two orbifold structures induced by the two actions of $\Gamma$ and $\Gamma'$. By \cite[Lemma 2.2]{Lange20} the two orbifold structures coincide since locally the actions of $\Gamma$ on $X$ and $\Gamma'$ on $X'$ are isometrically equivariant. Since both $X$ and $X'$ are simply connected, we deduce that they are both orbifold universal coverings of $M$ (see for instance \cite[Theorem 2.9]{Lange20}, originally proved in \cite{Thurston97}). By the universal property of the orbifold universal covering we get a diffeomorphism $F \colon X \to X'$ commuting with the projections $\pi$ and $\pi'$ on $M$.
By construction $F$ is also equivariant with respect to an isomorphism between $\Gamma$ and $\Gamma'$. Finally, again by \cite[Lemma 2.2]{Lange20} and the commutatation property $\pi' \circ F = \pi$, we get that $F$ is a local isometry and therefore a global one.
\end{proof}

  Strictly speaking, in view of the discussion above, the quotient metric space $M=\Gamma \backslash X$ should only be called the {\em support} of the orbispace;  we will continue to use the sloppy terminology {\em \textup{CAT}$(0)$-orbispace} both for the support and for the metric space with its uniformizing global chart $X \to M$, as far as the statements involved are clear.
For this reason, we  will say that two CAT$(0)$-orbispaces $M=\Gamma/X$ and $M' = \Gamma'/X'$ are:
\vspace{1mm}

\noindent -- {\em equivariantly homotopy equivalent}, if there exist an isomorphism $\varphi\colon \Gamma \to \Gamma'$ and a $\varphi$-equivariant homotopy equivalence\footnote{that is, a map satisfying  $F(gx) = \varphi(g)F(x)$ for every $x \in X$ and every $g \in \Gamma$, \linebreak admitting moreover a $\varphi^{-1}$-equivariant homotopy inverse $G: X' \to X$ (i.e. such that $F \circ G$ and $G\circ F$ are homotopic to the identity through, respectively, $\Gamma$-equivariant and $\Gamma'$-equivariant homotopies).}
$F\colon X \to X'$;   

\noindent -- {\em isometric as orbispaces} if, moreover, the map $F$ is an isometry.

\vspace{1mm}
 \subsection{Lie isometry groups}
\label{subsection-lie}
${}$

\noindent In general, the full isometry group Isom$(X)$ of a CAT$(0)$- space is not a Lie group, for instance in the case of regular trees.
When a CAT$(0)$-space $X$ admits a uniform lattice $\Gamma$ then Isom$(X)$ is known to have more structure, as proved by P.-E.Caprace and N.Monod:

\begin{prop}[\textup{\cite[Thm.1.6 \& Add.1.8]{CM09b}, \cite[Cor.3.12]{CM09a}}]
	\label{prop-CM-decomposition}${}$
	
\noindent	Let $X$ be a proper, geodesically complete, $\textup{CAT}(0)$-space, admitting a uniform lattice.
Then $X$ splits isometrically as $M \times \mathbb{R}^n \times N$, where $M$ is a symmetric space of noncompact type and ${\mathcal D}:=\textup{Isom}(N)$ is totally disconnected. \\ Moreover 
	$$\textup{Isom}(X) \cong {\mathcal S}   \times  {\mathcal E}_n\times {\mathcal D}$$
where ${\mathcal S}$ is a semi-simple Lie group with trivial center and without compact factors, and ${\mathcal E}_n =\textup{Isom}(\mathbb{R}^n)$.	
\end{prop}

\noindent In particular, under the assumptions of Proposition \ref{prop-CM-decomposition}, the group  $\textup{Isom}(X)$ is a Lie group if and only if the factor ${\mathcal D}$ is discrete.
On the other hand,  Isom$(X)$  is always a Lie group when $X$ is a CAT$(0)$-homology manifold:

\begin{prop}
	\label{cor-homology-Lie}
	Let $X$ be a locally compact, \textup{CAT}$(0)$-homology manifold. Then $\textup{Isom}(X)$ is a Lie group.
\end{prop}

\begin{proof}
	If the dimension of $X$ is smaller than or equal to $2$ then $X$ is a topological manifold. Otherwise we can find a Isom$(X)$-invariant, locally finite subset $E$ of $X$ such that $X\setminus E$ is a connected topological manifold (see \cite[Theorem 1.2]{LN-finale-18}). The Hausdorff dimension of $X\setminus E$ is at most the Hausdorff dimension of $X$, which coincides with its topological dimension, say $n$, as recalled at the beginning of Section \ref{sec-CAT}; therefore, $\text{HD}(X\setminus E) = n$.\\
The classical work about the Hilbert-Smith Conjecture implies that the group Isom$(X)$ is a Lie group if and only if it does not contain a subgroup isomorphic to some $p$-adic group ${\mathbb Z}_p$, see for instance \cite{lee}. 
Now, a theorem of Repov\c{s} and \c{S}\c{c}epin 
asserts that ${\mathbb Z}_p$ cannot act effectively by Lipschitz homeomorphisms on any finite dimensional Riemannian manifold.
The proof in \cite{RS97} adapts perfectly to our setting; indeed, if we suppose that the $p$-adic group acts on $X$ by isometries we can apply \cite{Yan60} as in \cite{RS97} to deduce that $\text{HD}(X\setminus E) \geq n+2$, giving a contradiction.
\end{proof}

\vspace{1mm}
\subsection{The packing condition and Margulis' Lemma}
\label{subsection-packing-margulis}
${}$

\noindent Let $X$ be a metric space and $r>0$.
A subset $Y$ of $X$ is called {\em $r$-separated} if $d(y,y') > r$ for all $y,y'\in Y$. Given $x\in X$ and $0<r\leq R$ we denote by Pack$(\overline{B}(x,R), r)$ the maximal cardinality of a $2r$-separated subset of $\overline{B}(x,R)$. Moreover we denote by Pack$(R,r)$ the supremum of Pack$(\overline{B}(x,R), r)$ among all points of $X$. 
Given $P_0,r_0 > 0$ we say that $X$ satisfies the $P_0$-packing condtion at scale $r_0$ (or that it is $(P_0,r_0)$-packed, for short) if Pack$(3r_0,r_0) \leq P_0$. \linebreak
 We will simply say that {\em $X$ is packed} if it satisfies a $P_0$-packing condition at scale $r_0$, for some $P_0,r_0>0$. \\
The packing condition should be thought as a metric, weak replacement  of 
a Ricci curvature lower bound.
Actually, by Bishop-Gromov's Theorem, for a $n$-dimensional Riemannian manifold a lower bound on the Ricci curvature $\text{Ric}_X \geq -(n-1)\kappa$, $\kappa \geq 0$, implies a uniform estimate of the packing function at any fixed scale $r_0$, that is 
\begin{equation}\label{eq-bishop}
\text{Pack}(3r_0,r_0) \leq \frac{v_{{\mathbb H}^n_{\kappa}}(3r_0)}{v_{{\mathbb H}^n_{\kappa}}(r_0)}
\end{equation}
where $v_{{\mathbb H}^n_{\kappa}}(r)$ is the volume of a ball of radius $r$ in the $n$-dimensional space form with constant curvature $-\kappa$.\\
  Also remark that every metric space admitting a cocompact action is packed (for some $P_0, r_0$), see the proof of \cite[Lemma 5.4]{Cav21ter}.
\vspace{1mm}

\noindent The packing condition  has many interesting geometric consequences for complete, geodesically complete CAT$(0)$-spaces, as showed in \cite{CavS20bis}, \cite{CavS20} \cite{Cav21bis},   \cite{Cav21}, \cite{CS22}, \cite{CS23}.  
 Here, we just recall the following uniform bounds  on the canonical measure of balls, and of the entropy and dimension: 

\begin{prop}[\textup{\cite[Theorem 4.9]{CavS20}}]
\label{prop-packing}
${}$
	
\noindent	Let $X$ be a complete, geodesically complete, $(P_0,r_0)$-packed, $\textup{CAT}(0)$-space. 
Then $X$ is proper, and 
\begin{itemize}[leftmargin=8mm] 
	 	\item[(i)] there exist functions $v,V \colon (0,+\infty) \to (0,+\infty)$ depending only on $P_0,r_0$ such that for all $x\in X$ and $R> 0$ we have
$$v(R) \leq \mu_X(\overline{B}(x,R)) \leq V(R);$$
	 \item[(ii)] the entropy of $X$ is bounded above in terms of $P_0$ and $r_0$, namely
$$\textup{Ent}(X) 
\leq \frac{\log (1+P_0)}{r_0};$$
\item[(iii)] the dimension of $X$ is at most $n_0 := P_0/2$.
\end{itemize}
\end{prop}


\noindent Recall that the  ({\em packing}, or {\em covering}) {\em entropy} of a proper CAT$(0)$-space $X$ is the asymptotic invariant defined as  
 $$\textup{Ent}(X) = \limsup_{R\to +\infty}\frac{1}{R}\log (\text{Pack}(R,r))$$
and does not depend neither on the point $x$ (by the triangular inequality) nor on the small, chosen radius $r$ (by \cite[Corollary 4.8]{CavS20}).
It can be equivalently defined as the critical exponent of any uniform lattice $\Gamma$ of $X$, as proved in  \cite{Cav21}), that is

\begin{equation}
	\label{defin:entropy-groups}
	\textup{Ent}(X)  = \lim_{R\to +\infty} \frac{1}{R} \log(\#( \Gamma x \cap B(x,R)))
\end{equation}


\vspace{2mm}

Let now $\Gamma$ be a subgroup of Isom$(X)$. For $x\in X$ and $r\geq 0$ we set 
\begin{equation}	
	\label{defsigma} 	
	\overline{\Sigma}_r(\Gamma,X,x) := \lbrace g\in \Gamma \text{ s.t. } d(x,gx) \leq r\rbrace
\end{equation}

\vspace{-6mm}

\begin{equation}	
	\label{defgamma} 	
	\overline{\Gamma}_r(X,x) := \langle \overline{\Sigma}_r(\Gamma, X,x) \rangle
\end{equation}
The groups  $\overline{\Gamma}_r(X,x)$ are sometimes called  {\em ``almost stabilizers''};  when the context is clear we will simply write $\overline{\Sigma}_r(\Gamma, x)$ or $\overline{\Sigma}_r(x)$ and $\overline{\Gamma}_r(x)$.\\
Notice that, since $X$ is assumed to be proper and $\Gamma$ acts by isometries, a subgroup $\Gamma$ is \emph{discrete} if and only if
 the orbit $ \Gamma x$ is discrete and $\text{Stab}_\Gamma (x)$ is finite for some (or, equivalently, for all) $x\in X$.  This is in turn equivalent to asking that the subsets $\overline{\Sigma}_r(x)$ are finite for all $x\in X$ and all $r\geq 0$.\\
The following remarkable version of the Margulis' Lemma, due to Breuillard-Green-Tao,  is another important consequence of a packing condition at some fixed scale. It clarifies the structure of the almost stabilizers   $\overline{\Gamma}_r(x)$ for small $r$. \linebreak We decline it for geodesically complete CAT$(0)$-spaces:

\begin{prop}[\textup{\cite[Corollary 11.17]{BGT11}}]
	\label{prop-Margulis-nilpotent}${}$\\
	Given $P_0,r_0 >  0$, there exists $\varepsilon_0 = \varepsilon_0(P_0,r_0) > 0$ such that the following holds.
	Let $X$ be a proper, geodesically complete, $(P_0,r_0)$-packed, $\textup{CAT}(0)$-space and let $\Gamma$ be a discrete subgroup of $\textup{Isom}(X)$: then, for every $x\in X$ and  every $0\leq \varepsilon \leq \varepsilon_0$, the almost stabilizer $\overline{\Gamma}_{\varepsilon}(x)$ is virtually nilpotent.
\end{prop}

\noindent We will often refer to the constant $\varepsilon_0=\varepsilon_0 (P_0,r_0)$ as the {\em Margulis' constant}.
The conclusion of Proposition \ref{prop-Margulis-nilpotent} can be improved for cocompact groups, as in this case the group $\overline{\Gamma}_{\varepsilon}(x)$ is {\em virtually abelian}   (cp. \cite[Theorem II.7.8]{BH09}; indeed, a cocompact group of a CAT$(0)$-space is always semisimple).



\vspace{1mm}
\subsection{Discrete, virtually abelian groups}
${}$ 

\noindent Recall that the (abelian, or Pr\"ufer)  {\em rank} 
of an abelian group $A$, denoted $\text{rk} (A)$,   is the maximal cardinality of a subset $S \subset A$
of $\mathbb{Z}$-linear independent elements. The rank of a {\em virtually abelian group} $G$, rk$(G)$, is the rank of any free abelian subgroup $A$ of finite index in $G$ (notice that if $A'$ is a finite index subgroup of an abelian group $A$, then $A$ and $A'$ have same rank, so rk$(G)$ is well defined).

%

\noindent Among CAT$(0)$-spaces, the Euclidean space ${\mathbb R}^k$ and its discrete groups play a special role. In the following, we will denote by ${\mathcal E}_k$ the group of isometries of $\mathbb{R}^k$, and by $\textup{Transl}({\mathbb R}^k)$  the normal subgroup of all translations.\\
A \emph{crystallographic group} is a discrete, cocompact group $G$ of isometries of   $\mathbb{R}^k$. The simplest and most important of them, in view of Bieberbach's Theorem, are {\em Euclidean lattices}: i.e. 
free abelian crystallographic groups.
It is well known that a lattice must act by translations on $\mathbb{R}^k$ (see for instance \cite{farkas}); so, alternatively, a lattice $\mathcal{L}$ can be seen as the set of linear combinations with integer coefficients of $k$ independent vectors $b_1, \ldots, b_k$ (we will make no difference between a lattice and this representation); the integer $k$ is  called the {\em rank} of the lattice.
%
%
%
	For our purposes, the content of the famous Bieberbach's Theorems can be stated as follows:
	\begin{prop}[Bieberbach's Theorem]
		\label{prop-Bieberbach}${}$

\noindent There exists $J(k)$, only depending on $k$, such that the following holds true.
	For any crystallographic group $G$ of $\mathbb{R}^k$ 
	the  subgroup ${\mathcal L}(G) = G \cap \textup{Transl}({\mathbb R}^k)$ is a normal subgroup of index at most $J(k)$, in particular a lattice. 
	\end{prop}
\noindent	The subgroup $\mathcal{L}(G)$ is called the \emph{maximal lattice} of $G$.

\noindent An almost immediate consequence of Bieberbach' theorem is the following lemma, which will be crucial in characterizing the limits of collapsing sequences in Section \ref{sub-convergencecollapsing}. For this we set  $r_k:=  \sqrt{2 \sin \left( \frac{\pi}{J(k) } \right)}$.

\begin{lemma}\label{lemma-bieber}
Let $G$ be  a  crystallographic group  of $\mathbb{R}^k$, and let  $0<r <r_k$. 
If $g \in G$ moves all points of  $B_{\mathbb{R}^k} (O, \frac{1}{r})$ less than $r$,    then $g$ is a translation.
\end{lemma}

 \begin{proof} 
Let $g(x)=Ax+b$ be any nontrivial isometry of $G$, with $A \in O(k)$. If $A= \text{Id}$   there is nothing to prove. 
Otherwise Proposition \ref{prop-Bieberbach} implies that $g^j$ is a translation for some $ j \leq J(k)$. Since $g^j (x) = A^j x + \sum_{i=0}^{j-1} A^i b$, we deduce that $A^j = \text{Id}$. 
Let us consider the (nontrivial) subspace $V$ orthogonal to $\text{Fix}(A)$.	
We can decompose $V$ in the ortogonal sum of a subspace $V_0$, on which $A$ acts as a  reflection, and of some planes $V_i$, each of  which $A$ acts on as a nontrivial rotation with angle $\vartheta_i$. Since $A$ has order at most $J(k)$, we have $\vartheta_i \geq \frac{2\pi}{J(k)}$.\\
If $V_0$ is not trivial then take a vector $v_0$ of length $1$ in $V_0$. Denote by $b_0 = \lambda v_0$ the projection of $b$ on the line spanned by $v_0$.
For all $t\in \mathbb{R}$ it holds
$$d(g(t v_0), t v_0) = \Vert -2t v_0 +  b \Vert \geq |\lambda-2t|.$$
So, for $r < \sqrt{2}$, we can always find $t_r$ with $\vert t_r \vert < \frac{1}{r}$ such that $|\lambda-2t_r| > r$; but the point $t_r v_0$ contradicts the assumptions on $g$.\\
On the other hand, if   $V_0$ is trivial, then there is at least a plane $V_i$ on which $A$ acts as a nontrivial rotation, say $V_1$.
Let $b_1$ be the projection of $ b$ on $V_1$. \\ We choose $v_1 \in V_1$ of norm $1$ such that $b_1 = \lambda (A v_1 -  v_1)$ for some $\lambda \geq 0$. \\
The length of $A v_1 -  v_1$ is at least 
$2 \sin \left( \frac{\pi}{J(k) } \right)$, because $\vartheta_1 \geq \frac{2\pi}{J(k)}$, therefore
$d(g(t v_1), tv_1) \geq  2t \sin \left( \frac{\pi}{J(k) } \right)  $
  $\forall t\geq 0$. Again, if $r < \sqrt{2  \sin \left( \frac{\pi}{J(k) } \right)}$ 
we can  find $t_r$ with $\vert t_r \vert \leq \frac{1}{r}$ such that  $t_r v_1$ contradicts the assumptions on $g$.
\end{proof}



\vspace{1mm}
\subsection{Finiteness results}
\label{subsec-systole}
${}$

\noindent The {\em systole} (or {\em minimal displacement}) and the {\em free-systole} of a discrete group $\Gamma$ {\em at a point} $x\in X$ are defined respectively as
$$\text{sys}(\Gamma,x) := \inf_{g\in \Gamma^*} d(x,gx) , 
\qquad \text{sys}^\diamond(\Gamma,x) := \inf_{g\in \Gamma^\ast \setminus \Gamma^\diamond} d(x,gx),$$
where $\Gamma^* = \Gamma \setminus \lbrace \text{id} \rbrace$ and $\Gamma^\diamond $ is the subset of all elliptic isometries of $\Gamma$.\\
Accordingly, the  {\em (global) systole} and the  {\em free-systole of} $\Gamma$ are  defined as 
$$\text{sys}(\Gamma,X) = \inf_{x\in X}\text{sys}(\Gamma,x), \qquad 
\text{sys}^\diamond(\Gamma,X) = \inf_{x\in X}\text{sys}^\diamond(\Gamma,x).$$
as opposite to the {\em diastole}   and the  {\em free-diastole} of $\Gamma$, which  are  
$$\text{dias}(\Gamma,X) = \sup_{x\in X}\text{sys}(\Gamma,x) , \qquad 	
\text{dias}^\diamond(\Gamma,X) = \sup_{x\in X}\text{sys}^\diamond(\Gamma,x).$$	
Notice that the action of $\Gamma$ is nonsingular if and only if 
$\text{dias}(\Gamma,X)>0$.

\noindent While obviuosly one always has, by definition, 
$$\text{sys}(\Gamma,X) \leq  \text{sys}^\diamond(\Gamma,X) \leq  \text{dias}^\diamond(\Gamma,X)$$
$$\text{sys}(\Gamma,X) \leq  \text{dias}(\Gamma,X) \leq  \text{dias}^\diamond(\Gamma,X)$$
the free systole and the diastole are  quantitatively equivalent (for small values) for nonsingular, cocompact actions on packed CAT$(0)$-spaces with bounded codiameter.
%
%
More precisely: 
%

\begin{prop}[Theorem 3.1, \cite{CS23}]
\label{prop-vol-sys-dias}
Given $P_0,r_0$ and $D_0$,  there exists \linebreak $A = A(P_0,r_0) > 0$, $B = B(P_0,D_0) > 1$ such that the following holds true:
for any proper, geodesically complete, $(P_0,r_0)$-packed, $\textup{CAT}(0)$-space $X$, and any discrete, nonsingular, $D_0$-cocompact group  $\Gamma<\textup{Isom}(X)$  we have
	$$B ^{\hspace{0,5mm} -1/\textup{dias}(\Gamma, X)} 
	\leq \textup{sys}^\diamond(\Gamma, X)
	\leq  A \cdot \textup{dias}(\Gamma, X)
	$$
provided that the free systole and the diastole are both smaller than $\min\{ r_0, \varepsilon_0\}$.
\noindent	(Here, $\varepsilon_0$ is the Margulis' constant given by Proposition \ref{prop-Margulis-nilpotent}). 
\end{prop}

\noindent This equivalence will be useful when studying the  convergence of non-collapsing sequences of groups with elliptic elements. 
\vspace{1mm}

We conclude recalling two of the main findings of \cite{CS23}, which we will  be crucial in our convergence results.
The first is a general finiteness theorem for uniformily packed,  $D_0$-cocompact CAT$(0)$-lattices $\Gamma$, up to abstract isomorphism.
The second one is more subtle and also requires a uniform lower bound of the diastole: it bounds the number of such groups $\Gamma$ up to isomorphism of marked groups.
In order to  state it correctly, recall that a {\em marked group} is a group $\Gamma$ endowed with a generating set $\Sigma$;
two marked groups  $(\Gamma, \Sigma)$ and $(\Gamma', \Sigma')$ are said to be {\em equivalent} if there exists a group isomorphism   $\phi: \Gamma \rightarrow \Gamma'$ such that $\phi(\Sigma)=\Sigma'$.
It is well known (cp. for instance \cite{Ser80}, \cite{Gro78b})
that if $X$ is a geodesic metric space and $\Gamma < \textup{Isom}(X)$ is discrete and $D_0$-cocompact,  then for all $D \geq D_0$   the subset  
\begin{equation}	
	\label{defshort} 	
	\overline{\Sigma}_{2D}(x) := \lbrace g\in \Gamma \text{ s.t. } d(x,gx) \leq 2D\rbrace
\end{equation}
is a  generating set for $\Gamma$, that is 
$\overline{\Gamma}_{2D}(x)=\Gamma$, for every $x\in X$ 
(recall the definition \eqref{defgamma}) of $\overline{\Gamma}_{r}(x)$):
we call this a {\em $2D$-short generating set  of $\Gamma$  at  $x$.}

\begin{theo}[{\cite[Theorem A \& Prop.5.1]{CS23}}]
	\label{theo-finiteness} ${}$
Let $P_0,r_0$ and $D_0$ be fixed.  \linebreak 
There exist only finitely many nonsingular $D_0$-cocompact lattices $(\Gamma,X)$ \linebreak of proper, geodesically complete, \textup{CAT}$(0)$-spaces which are $(P_0,r_0)$-packed, up to group isomorphisms.\\
Moreover, for every fixed  $D \geq D_0$ and $s>0$, there exist only finitely many marked groups $(\Gamma, \Sigma)$, up to equivalence of marked groups, such that:
\vspace{-1mm}

	\begin{itemize}[leftmargin=4mm] 
		\item[--] $\Gamma$ is a $D_0$-cocompact lattice of  a proper, geodesically complete, $(P_0,r_0)$-packed, \textup{CAT}$(0)$-space $X$ satisfying  $\textup{sys} (\Gamma,x) \geq s$ for some $x \in X$,   
		\item[--]   $\Sigma=\overline{\Sigma}_{2D}(x)$ is a $2D$-short generating set of $\Gamma$ at $x$.
	\end{itemize}
\end{theo}

\section{Splitting under collapsing}
\label{sec:recall}

\subsection{Space splitting}
\label{subsec-splitting}
${}$

\noindent We  record here one of the main results of \cite{CS23}:  if a uniform  lattice $\Gamma$ of a packed, CAT$(0)$-space $X$ is sufficiently collapsed (that is, the free-systole is sufficiently small), then $X$ splits a non-trivial Euclidean factor. 

\noindent Recall that the dimension of any (proper, geodesically complete) $(P_0, r_0)$-packed  CAT$(0)$-space  $X$ is bounded above by $n_0 = P_0/2$, by Proposition \ref{prop-packing}. \linebreak Recall also the Margulis's constant $ \varepsilon_0$  given by Proposition \ref{prop-Margulis-nilpotent}, which depends only on the packing constants $P_0, r_0$. Then, we set 
\begin{equation}\label{eqJ0} J_0 := \max_{k\in \lbrace 0,\ldots,n_0\rbrace}J(k) +1 \end{equation}
where  $J(k)$ is the constant appearing in Bieberbach's theorem  (Proposition \ref{prop-Bieberbach}),  bounding the index of the maximal lattice of any crystallographic group in dimension $k$.
Notice that $J_0$ depends only on $n_0$, so ultimately only on $P_0$.  \linebreak
Finally, for a discrete subgroup  $\Gamma < \textup{Isom}(X)$, 
recall the definition (\ref{defgamma}) of almost-stabilizer  $\overline{\Gamma}_{r} (x) < \Gamma$ generated by  $\overline{\Sigma}_r (x)$ given in Section \ref{subsection-isometries}.

\begin{theo}[see  \cite{CS23},  Splitting Theorem 4.1 and its proof]
	\label{theo-splitting-weak}${}$
	
\noindent	Given $P_0,r_0$ and $D_0$, 
	there exists a   function 
	$\sigma_{P_0,r_0,D_0}: (0,\varepsilon_0] \rightarrow  (0,\varepsilon_0]$   (depending only on the parameters $P_0,r_0,D_0$) such that the following holds. 
	Let $X$ be a proper, geodesically complete, $(P_0,r_0)$-packed, $\textup{CAT}(0)$-space,   and $\Gamma < \textup{Isom}(X)$ be discrete and $D_0$-cocompact. 
	For every chosen $\varepsilon \in (0, \varepsilon_0]$, if there exists $x_0\in X$ such that $\textup{rk}(\overline{\Gamma}_\sigma(x_0)) \geq 1$, where $\sigma = \sigma_{P_0,r_0,D_0}(\varepsilon)$, then:
	\begin{itemize}[leftmargin=9mm] 
		\item[(i)] the space $X$ splits isometrically  and $\Gamma$-invariantly as $Y \times \mathbb{R}^k$, with $k\geq 1$; 
		\item[(ii)] there exists ${\varepsilon^\ast} \in (\sigma_{P_0,r_0,D_0} (\varepsilon), \varepsilon)$ 
		such that the rank of the virtually abelian subgroups $\overline{\Gamma}_{{\varepsilon^\ast}}(x)$ is  exactly $k$, for all $x\in X$; 
		\item[(iii)] for every $x\in X$ there exists $y\in Y$ such that $\overline{\Gamma}_{\varepsilon^\ast}(x)$ preserves $\lbrace y \rbrace \times \mathbb{R}^k$;   
		\item[(iv)]    for every $x\in X$  the projection 
		of $\overline{\Gamma}_{\varepsilon^\ast}(x)$ on $\textup{Isom}(\mathbb{R}^k)$ is a crystallographic group, whose maximal lattice is generated by the projection of a subset $\Sigma \subset \Sigma_{4J_0\cdot{\varepsilon^\ast}}(x)$; 
		\item[(v)]   for every $x\in X$  the closure of the projection of $\overline{\Gamma}_{\varepsilon^\ast}(x)$ on $\textup{Isom}(Y)$ is compact and totally disconnected;
		\item[(vi)]  for every $x \in X$ the group $\overline{\Gamma}_{\varepsilon^*}(x)$ contains a  finite index,  free abelian subgroup $A$ of rank $k$ which is commensurated  in $\Gamma$ (hence,  $\overline{\Gamma}_{\varepsilon^*}(x)$ is itself commensurated in $\Gamma$).
	\end{itemize}			
In particular, the above properties  hold  if $\textup{sys}^\diamond(\Gamma,X) \leq \sigma_{P_0,r_0,D_0}(\varepsilon)$.
\end{theo}

\noindent Here, by {\em $\Gamma$-invariant splitting} we mean that every isometry of $\Gamma$ preserves the product decomposition. By \cite[Proposition I.5.3.(4)]{BH09} we can see $\Gamma$ as a subgroup of $\textup{Isom}(Y) \times \textup{Isom}(\mathbb{R}^k)$. In particular it is meaningful to talk about the projection of $A$ on $\text{Isom}(Y)$ and $\text{Isom}(\mathbb{R}^k)$. Recall that a subgroup   $G < \Gamma$ is said to be {\em commensurated in $\Gamma$} if   the groups $G$ and $\gamma G \gamma^{-1}$ are commensurable in $\Gamma$ for every $\gamma \in \Gamma$ (i.e.  the intersection $G  \cap \gamma G \gamma^{-1}$   has finite index in  $G$  for every $\gamma \in \Gamma$). 
\vspace{1mm}

%

 Let us define, for the following, the constant (only depending on $P_0,r_0,D_0$)
\begin{equation}
\label{sigma0ast}
\sigma_0 ^\ast :=  \sigma_{P_0,r_0,D_0} (\varepsilon_0)
\end{equation}
Assertion (iii) yields, for any $x\in X$, a slice $\lbrace y \rbrace \times \mathbb{R}^k$ preserved by $\overline{\Gamma}_{\sigma_0^\ast}(x)$ (as $\sigma_0^\ast <  \varepsilon_0^\ast < \varepsilon_0$),   provided that $\textup{sys}^\diamond(\Gamma,X) \leq \sigma_0^\ast$ 
 so  that  $X$ splits as $Y \times \mathbb{R}^k$. For our convergence theorems, we need to strengthen this  conclusion  and to show   that, for at least one specific point $x_0$,  the slice   can be chosen to be the one passing through $x_0$. This is proved by the following:

\begin{prop}
	\label{prop-minimal-close}
	Let $X$ be a  proper, geodesically complete, $(P_0,r_0)$-packed   $\textup{CAT}(0)$-space,   and let  $\Gamma$  be a  discrete, $D_0$-cocompact subgroup of $\textup{Isom}(X)$.
	Assume that $\textup{sys}^\diamond(\Gamma,X) \leq \sigma_0^\ast$, so that $X$ splits as $Y\times \mathbb{R}^k$ by Theorem \ref{theo-splitting-weak}. 
	 Then there exists $x_0 = (y_0,  v)\in X$ such that $\overline{\Gamma}_{\sigma_0 ^\ast}(x_0)$ preserves $\lbrace y_0 \rbrace \times \mathbb{R}^k$.
\end{prop}


\noindent The proof is based on a maximality argument similar to  \cite[Theorem 3.1]{CS22}.\linebreak
We just need a very basic fact:
\begin{lemma}[\textup{\cite[Lemma 3.3]{CS22}}]
	\label{lemma-basic-containment}
For every $x\in X$ and $r > 0$ there exists an open set $U \ni x$ such that $\overline{\Sigma}_r(y)\subseteq \overline{\Sigma}_r(x)$ for all $y\in U$.
\end{lemma}	

\begin{proof}[Proof of Proposition \ref{prop-minimal-close}]
	We introduce a partial order on $X$ defined by 
$$x \preceq x' \mbox{  if and only if } \overline{\Sigma}_{\sigma_0 ^\ast}(x) \subseteq \overline{\Sigma}_{\sigma_0 ^\ast}(x').$$
We can show that there is a maximal element as in the proof of \cite[Theorem 3.1, Step 3]{CS22}, which we report here for the reader's convenience: Lemma \ref{lemma-basic-containment} implies that the function $x\mapsto \# \overline{\Sigma}_{\sigma_0 ^\ast}(x)$ is upper semicontinuous and clearly $\Gamma$-invariant, so (by cocompactness) it has a maximum and it is enough to take a point $x$ where the maximum is realized. \\
 Apply Theorem \ref{theo-splitting-weak}.(iii) to the point $x$: there exists $y_0\in Y$ such that $\overline{\Gamma}_{\sigma_0 ^\ast}(x)$ preserves $\lbrace y_0 \rbrace \!\times\mathbb{R}^k$. If $x \in  \lbrace y_0 \rbrace \!\times \mathbb{R}^k$, then we set $x_0=x$ and
 there is nothing  more to prove. 
 Otherwise call $x_0$ the projection of $x$ on the closed, convex, $\overline{\Gamma}_{\sigma_0 ^\ast}(x)$-invariant subset $\lbrace y_0 \rbrace \!\times \mathbb{R}^k$. 
 Let $c\colon [0,d(x,x_0)] \to X$ be the geodesic $[x,x_0]$ and set 
	$$T=\sup\lbrace t \in [0,d(x,x_0)] \text{ s.t. } \overline{\Gamma}_{\sigma_0 ^\ast}(c(t)) = \overline{\Gamma}_{\sigma_0 ^\ast}(x)\rbrace.$$
By definition $T\geq 0$ and we claim that $T=d(x,x_0)$ and it is  a maximum.\\
	In fact, assume $\overline{\Gamma}_{\sigma_0 ^\ast}(c(t)) = \overline{\Gamma}_{\sigma_0 ^\ast}(x)$. 
So, $\overline{\Gamma}_{\sigma_0 ^\ast}(c(t + t')) \subseteq \overline{\Gamma}_{\sigma_0 ^\ast}(c(t)) = \overline{\Gamma}_{\sigma_0 ^\ast}(x)$ for all $t'>0$ small enough, by Lemma \ref{lemma-basic-containment}.
	On the other hand, for every $g \in \overline{\Gamma}_{\sigma_0 ^\ast}(x)$ we have $d(x_0, g x_0) \leq d(x,gx)$ (since $g$ acts on $Y$ fixing $y_0$).
Therefore, by the convexity of the displacement functions we deduce  that 
$\overline{\Gamma}_{\sigma_0 ^\ast}(x) \subseteq \overline{\Gamma}_{\sigma_0 ^\ast}(c(t+t'))$ too. 	
	This shows that the supremum defining $T$ is not realized, unless $T = d(x,x_0)$. Let now $0\leq t_n < T$ such that $t_n \to T$, so the points $c(t_n)$ converge to $c(T)$. By Lemma \ref{lemma-basic-containment} we get $\overline{\Sigma}_{\sigma_0 ^\ast}(c(t_n)) \subseteq \overline{\Sigma}_{\sigma_0 ^\ast}(c(T))$ for all $n$ big enough. But $\overline{\Sigma}_{\sigma_0 ^\ast}(c(t_n)) = \overline{\Sigma}_{\sigma_0 ^\ast}(x)$ is maximal for $\preceq$, therefore $\overline{\Sigma}_{\sigma_0 ^\ast}(x) = \overline{\Sigma}_{\sigma_0 ^\ast}(c(t_n)) = \overline{\Sigma}_{\sigma_0 ^\ast}(c(T))$ for all $n$ big enough. This implies that $T=d(x,x_0)$ is actually  a maximum.
Hence,
$\overline{\Gamma}_{\sigma_0 ^\ast}(x_0) = \overline{\Gamma}_{\sigma_0 ^\ast}(x)$,
and  $x_0$ is the point we are looking for.
\end{proof}

A consequence of Theorem \ref{theo-splitting-weak}
is the following control of the almost stabilizers, which will be useful in   studying  converging sequences.
Recall the function 	  $\sigma_{P_0,r_0,D_0}: (0,\varepsilon_0] \rightarrow  (0,\varepsilon_0]$ given by Theorem \ref{theo-splitting-weak}, where $\varepsilon_0$ is the Margulis constant, and the  constant $J_0$  introduced at the beginning of this section.
 Then:
 

	\begin{cor}
		\label{cor-0-rank}
Let $X$ be a proper, geodesically complete, $(P_0,r_0)$-packed $\textup{CAT}(0)$-space,   and let  $\Gamma$  be a  discrete, $D_0$-cocompact subgroup of $\textup{Isom}(X)$. 
Let $\sigma (\Gamma, X) \! :\!= \! \sigma_{P_0,r_0,D_0} (\varepsilon^\diamond)$
		be the constant obtained for $\varepsilon^\diamond \!= \!\min\left\{ \varepsilon_0, \frac{\textup{sys}^\diamond(\Gamma,X)}{4J_0} \right\}$.
		Then
\begin{itemize}[leftmargin=9mm] 
\item[(i)]   the almost stabilizers   $\overline{\Gamma}_{\sigma (\Gamma, X)}(x)$ are finite, for all $x\in X$; 
\item[(ii)]   there exists $x_0\in X$ such that $\overline{\Gamma}_{\sigma (\Gamma, X)}(x_0)$ fixes $x_0$.
\end{itemize}
 	\end{cor}
	
\noindent Observe that, since by (i) the group $\overline{\Gamma}_{\sigma(\Gamma,X)}(x)$ is finite, then it has a fixed point (as it acts on a CAT$(0)$-space, cp. \cite[Chapter II, Corollary 2.8]{BH09}). \linebreak Here we are saying  that for at least one specific point $x_0$ we have that $\overline{\Gamma}_{\sigma(\Gamma,X)}(x_0)$ fixes exactly $x_0$.

	\begin{proof}[Proof of Corollary \ref{cor-0-rank}]
  We first prove  (i), and set, for short, $\sigma= \sigma (\Gamma, X)$.
		 Suppose that there exists $x\in X$ such that $\overline{\Gamma}_{\sigma}(x)$ is not finite: since $\overline{\Gamma}_{\sigma}(x)$ is finitely generated and virtually abelian, then  it contains a hyperbolic isometry, so $\text{rk}(\overline{\Gamma}_{\sigma}(x )) \geq 1$.
 Theorem \ref{theo-splitting-weak} implies that $X$ splits isometrically and $\Gamma$-invariantly as $Y\times \mathbb{R}^k$, with $k\geq 1$, and that there exists  $\varepsilon^\ast \in \left(\sigma,  \varepsilon^\diamond \right)$ such  that $\overline{\Gamma}_{\varepsilon^\ast} (x)$ projects on $\text{Isom}(\mathbb{R}^k)$ as a crystallographic group. 
Moreover, we can find elements of $\overline{\Sigma}_{4J_0 \cdot \varepsilon^\ast}(x)$ that project to translations on $\text{Isom}(\mathbb{R}^k)$. \linebreak
But these elements are hyperbolic isometries of $\Gamma$ with translation length at most $4J_0\cdot \varepsilon^\ast < \textup{sys}^\diamond(\Gamma,X)$, a contradiction.\\
Assertion (ii)  follows from an argument similar to that of Proposition \ref{prop-minimal-close}:  
we consider the same partial order  $x \preceq x'  \Leftrightarrow \overline{\Sigma}_{\sigma}(x) \subseteq \overline{\Sigma}_{\sigma}(x')$, and find a maximal element $x$.
The group $\overline{\Gamma}_{\sigma}(x)$ being  finite,   the closed, convex set $\text{Fix}(\overline{\Gamma}_{\sigma}(x))$ is not empty. If $x\in \text{Fix}(\overline{\Gamma}_{\sigma}(x))$, then $x_0=x$ and there is nothing more to prove. Otherwise we call $x_0$ the projection of $x$ on $ \text{Fix}(\overline{\Gamma}_{\sigma}(x))$, consider again the geodesic $c\colon [0,d(x,x_0)] \to [x,x_0] \subset X$  and  we show as in Proposition \ref{prop-minimal-close} that 
$$ T :=\sup\lbrace t \in [0,d(x,x_0)] \text{ s.t. } \overline{\Gamma}_{\sigma}(c(t)) = \overline{\Gamma}_{\sigma}(x)\rbrace = d(x,x_0)$$
by Lemma \ref{lemma-basic-containment} and the convexity of the displacement function. Moreover, again by Lemma \ref{lemma-basic-containment} and by the maximality of $x$ it follows that  $T=d(x,x_0)$ is a maximum,
hence $\overline{\Gamma}_{\sigma}(x_0) = \overline{\Gamma}_{\sigma}(x)$, and   $x_0$ is the announced fixed point.
 \end{proof}
\vspace{1mm}

 Finally, we record here another corollary of the splitting theorem, which will be used  in studying the Euclidean factor of the limit of a sequence of collapsing actions in Section \ref{sub-characterization}${}$.
 
 \begin{theo}[\cite{CS23}, Renormalization Theorem E]
	\label{theo:renormalization} ${}$\\
	Given $P_0,r_0,D_0$, there exist $s_0 = s_0(P_0,r_0,D_0)>0$ and $\Delta_0 = \Delta_0(P_0,D_0)$ \linebreak such that the following holds.
	Let $\Gamma$ be a discrete and $D_0$-cocompact isometry group of 
	a proper, geodesically complete, $(P_0,r_0)$-packed $\textup{CAT}(0)$-space $X$:  
	then $\Gamma$ admits also a faithful, discrete, $\Delta_0$-cocompact action by isometries on a $\textup{CAT}(0)$-space $X'$   isometric to $X$,  such that 
	$\textup{sys}^\diamond(\Gamma, X') \geq s_0$.\\
	Moreover, the action of $\Gamma$ on $X$ is nonsingular if and only if the action on $X'$ is nonsingular.
\end{theo}

	\subsection{Group splitting: a counterexample}
	\label{sec-groupsplitting} ${}$ 
	
\noindent 	The Splitting Theorem \ref{theo-splitting-weak}
	might suggest that, under the assumption that the action is sufficiently collapsed, then also the group $\Gamma$ virtually splits a non-trivial free abelian subgroup. 
	Unfortunately, while this is always the case for non-positively curved Riemannian manifolds by the work of Eberlein \cite{Ebe83} (indeed, in that case it is always possible to split virtually a 
	free abelian subgroup whose rank coincides exactly with the rank of the maximal Euclidean factor of the space), this is no longer true for discrete, cocompact groups of  general, geodesically complete CAT$(0)$-spaces,	  as \cite[Example 1]{CM19}  shows. \linebreak More examples of this phenomenon can be found in \cite{LM21}. \\
	In this section, we present a basic example, constructed in \cite{LM21}, showing that, no matter how the action is collapsed, the virtual splitting cannot be proved under our hypotheses. 
	 


	\begin{ex}
		\label{ex-comm-not-normal}
		We construct a sequence of proper, geodesically complete, CAT$(0)$-spaces $X_j$, all $(P_0,r_0)$-packed for the same constants $P_0,r_0 >0$, and a sequence of discrete, nonsingular groups $\Gamma_j < \textup{Isom}(X_j)$ which are all $D_0$-cocompact for some $D_0>0$ such that:
		\begin{itemize}
			\item[(a)] each $\Gamma_j$ has no non-trivial, virtually normal, abelian subgroups;
			\item[(b)] $\text{sys}^\diamond(\Gamma_j,X_j) \to 0$ as $j\to +\infty$.
		\end{itemize}
		Recall that a subgroup $A$ of a group $\Gamma$ is said to be {\em virtually normal} if there exists a finite index subgroup $\Gamma'$ such that $A \cap \Gamma'$ is normal in $\Gamma'$.\\
		Observe that (b) gives strong restrictions to the spaces $X_j$ and the groups $\Gamma_j$ in virtue of Theorem \ref{theo-splitting-weak}; indeed, the spaces $X_j$ will split a Euclidean factor and the groups $\Gamma_j$ will have a commensurated non-trivial, abelian subgroup.\\
	On the other hand, the condition that $\Gamma$ virtually splits a free abelian factor of positive rank is equivalent to have a virtually normal, non trivial, free abelian subgroup, by \cite[Theorem 2]{CM19}; so, these groups $\Gamma_j$ do not split.
\vspace{1mm}	
	
\noindent		The idea behind the example is very simple. \\
We start as in \cite{LM21} from the lattices ${\mathcal L}_+, {\mathcal L}_-$ of $\mathbb{Z}^2=  \langle a,b \,\,|\,\, [a,b]=1 \rangle$
respectively generated by ${\mathcal B}_+= \{ a^2b^{-1},  ab^2 \}$ and by ${\mathcal B}_-= \{ a^2b,  a^{-1}b^2 \}$, and define the (torsion-free) HNN-extension 
		$$\Gamma = \langle a,b,t \,\,|\,\, [a,b]=1,\,\, ta^2b^{-1}t^{-1} = a^2b,\,\, tab^2t^{-1} = a^{-1}b^2 \rangle$$
that is 
$\Gamma=\mathbb{Z}^2 \ast\varphi$, where $\varphi: {\mathcal L}_+ \rightarrow {\mathcal L}_-$ is given by $\varphi(a^2b^{-1})=a^2b$ and $\varphi(ab^2)=a^{-1}b^2$.
We denote by $T$ its Bass-Serre tree: this is a locally finite tree, since $[\mathbb{Z}^2: {\mathcal L}_\pm] =5$. 
The group $\Gamma$ acts naturally on $T$ and on $\mathbb{R}^2$, with $a,b$ acting on the latter as orthogonal translations of length $1$ and $t$ as a rotation around the origin $O$ of angle $\arccos(\frac{3}{5})$. \\
Let us consider the diagonal action of $\Gamma$ on the (proper, geodesically complete) CAT$(0)$-space $X:=T \times  \mathbb{R}^2$, which is $(P_0,1)$-packed for some $P_0>0$. \\
First notice that this action 
is faithful. In fact, the pointwise stabilizer of the axis of $t$ on $T$ is $\bigcap_{n \in \mathbb Z} t^n\mathbb Z^2 t^{-n}$, which is trivial. Indeed, suppose that $w \in \bigcap_{n \in \mathbb Z} t^n\mathbb Z^2 t^{-n}$, that is $w = t^n z_n t^{-n}$ for some sequence $(z_n)_{n \in {\mathbb N}}$ in $\mathbb{Z}^2$. 
Then $d_{\mathbb{R}^2}(wO,O)=d_{\mathbb{R}^2}(t^nz_nO,O) = d_{\mathbb{R}^2}(z_nO,O)$ for every $n$, as $t$ fixes $O$.
By discreteness of the action of $\mathbb{Z}^2$ on $\mathbb{R}^2$ we can assume that $z_n=z\in \mathbb{Z}^2$ for infinitely many $n$'s; therefore, there exist integers $n \neq m$ such that $w = t^n z t^{-n} = t^m z t^{-m}$, i.e. $z = t^k z t^{-k}$ for some $k \neq 0$. This implies that $zO = t^kzO$, but since $t$ acts as a rotation by an angle which is an irrational multiple of $\pi$, then $t^k$ does not fix any point different from $O$. It follows that  $z$, and in turn $w$, is necessarily trivial.
\\
Moreover, it is straightforward to check that the action of $\Gamma$ on $X$ is discrete, while both the actions of $\Gamma$ on the two factors are not. \\
Finally, the action of $\Gamma$ on $X$ is nonsingular (since $\Gamma$ is torsion-free) and $1$-cocompact, since it is on both factors. \\
It remains to check that $\Gamma$ has no virtually normal, nontrivial, abelian subgroups. Suppose $\Gamma$ virtually normalizes an abelian subgroup $A$ of rank $k \ge 1$. By \cite[Theorem 2.(ii)]{CM19} the space $X = \mathbb{R}^2 \times T$ splits as $Y \times \mathbb{R}^k$ and a finite index subgroup $\Gamma'$ of $\Gamma$ splits as a product $\Gamma_Y \times \mathbb{Z}^k$, where $\Gamma_Y$ acts discretely on $Y$ and $\mathbb{Z}^k$ acts as a lattice on $\mathbb{R}^k$. 
By \cite[Theorem 1.1]{FL08} we deduce that either $Y=T$ (if $k=2$) or $Y=\mathbb{R} \times T$ (if $k=1$). The first case has to be excluded since in that case the projection of $\Gamma$ on $\mathbb{R}^2$ would be discrete, containing the discrete group $\Gamma_Y$ as finite index subgroup. In the second case we would find a finite index subgroup of the projection of $\Gamma$ on $\mathbb{R}^2$ that preserves a line, which is clearly not possible. \\
	We now take $X_j = \frac{1}{j}\cdot \mathbb{R}^2 \times T$ and define $\Gamma_j := \Gamma$ with its natural action on $X_j$ induced by the action of $\Gamma$ on $X$. The sequence of groups $\Gamma_j$ acting on $X_j$ clearly satisfies the conditions (a) and (b).
\end{ex}
	\noindent The behaviour of this sequence will be further studied in Section \ref{sub-convergencecollapsing},  Example \ref{rem-totdiscvsdicrete}. 
\vspace{1mm}
	
\subsection{Group splitting: when the isometry group is Lie}
\label{sec-lie}
${}$\\
The purpose of this section is to prove  that  a uniform lattice  $\Gamma$ of a $\textup{CAT}(0)$-space $X$ with a nontrivial Euclidean factor ${\mathbb R}^k$  virtually splits a  free abelian subgroup of rank $k$,   provided that  Isom$(X)$ is  a Lie group. 
By Proposition 	\ref{cor-homology-Lie}, this holds, in particular,  for  all CAT$(0)$-homology manifolds.

\noindent The group splitting stems almost immediately from \cite{CM19}. However,  for our convergence results,
we need a more precise result, that is to bound  the index in $\Gamma$ of the subgroup  $\check \Gamma$  which splits ${\mathbb Z}^k$.

 \begin{theo}
	\label{theo-group-splitting-Lie}
Given $P_0,r_0,D_0$, there exist $I_0$ and  
  $\sigma_0 ^\ast>0$ (the same as in Proposition \ref{prop-minimal-close}),
only depending on $P_0,r_0,D_0$,
such that the following holds.\\
	Let $X$ be a proper, geodesically complete, $(P_0,r_0)$-packed, $\textup{CAT}(0)$-space whose isometry group is a Lie group, and let $X = Y \times \mathbb{R}^n$ be the splitting of the maximal Euclidean factor of $X$. Then for every discrete, $D_0$-cocompact group  $\Gamma < \textup{Isom}(X)$ we have: 
\begin{itemize}
		\item[(i)] $\Gamma$   has a normal, finite index subgroup $\check \Gamma$  which splits as $\check \Gamma = \check \Gamma_Y \times \mathbb{Z}^n$;  
		\item[(ii)]  the projections of $\Gamma$ on $\textup{Isom}(Y)$ and $\textup{Isom}(\mathbb{R}^n)$ are discrete, as well as the factors  $\check \Gamma_Y$ and $\mathbb{Z}^n$ of    $\check \Gamma$;
		\item[(iii)] $\textup{sys}^\diamond(\Gamma_Y,Y) \geq \sigma_0^\ast$.
	\end{itemize}
Moreover if $\Gamma$ is nonsingular then 	 $\check \Gamma$ can be chosen of index  $[\Gamma: \check \Gamma] \leq I_0$  and $\check{\Gamma}_Y$ is nonsingular.
As a result,  if $\Gamma$ is nonsingular, $M = \Gamma \backslash X $ is the quotient by a  finite group of isometries $\Lambda \cong \Gamma / \check \Gamma$   with  $| \Lambda|  \leq  I_0$ of a  space $\check{M} $  which   splits isometrically as   $ \check{N} \times \mathbb{T}^{n}$, where
  $\check N=\check \Gamma_Y \backslash Y$ and $\mathbb{T}^{n}$ is a flat $n$-dimensional torus.
\end{theo}

\noindent We also say that $\Gamma$ and  $M$ {\em virtually split} as $\check \Gamma_Y \times \mathbb{Z}^{n}$ and $\check N \times \mathbb{T}^{n}$, respectively.

 \begin{proof} By Proposition \ref{prop-CM-decomposition}, the space $X$ splits isometrically as $M\times  N \times \mathbb{R}^n $ and $\text{Isom}(X) \cong {\mathcal S} \times {\mathcal D} \times {\mathcal E}_n $ where ${\mathcal S} \cong \text{Isom}(M)$ is a semi-simple Lie group with trivial center and no compact factors, ${\mathcal D} \cong \text{Isom}(N)$ is totally disconnected and ${\mathcal E}_n\cong \text{Isom}(\mathbb{R}^n)$. Since $\text{Isom}(X)$ is assumed to be a Lie group then ${\mathcal D}$ must be discrete. Let  $\Gamma_{{\mathcal S}},  \Gamma_{{\mathcal D}}$ and $\Gamma_{{\mathcal E}_n}$ be the projections of $\Gamma $ on ${\mathcal S}, {\mathcal D}$ and ${\mathcal E}_n$    respectively. By \cite[Theorem 2(i)]{CM19} there exists a free abelian subgroup $A$ of $\Gamma$ of rank $n$ which is commensurated in $\Gamma$, and its projection $A_{{\mathcal E}_n}$ on ${\mathcal E}_n$ is a crystallographic group.
	Arguing as in the proof of that theorem, we also deduce that the image $A_{{\mathcal S}}$ of $A$ on ${\mathcal S}$ is finite by Borel density. \linebreak On the other hand, the projection $A_{\mathcal D}$ of $A$ on ${\mathcal D}$ is an abelian, commensurated subgroup of $\Gamma_{\mathcal D}$. 
Since $\Gamma_{\mathcal D}$  acts  cocompactly, hence minimally, on the factor $N$, and every irreducible factor of $N$ is not Euclidean, Proposition 4 of \cite{CM19} implies that $A_{\mathcal D}$ either fixes a point or acts minimally on ${\mathcal D}$; but an abelian group cannot act minimally on a proper, non-trivial, CAT$(0)$-space without Euclidean factors, therefore $A_{\mathcal D}$ fixes a point in $N$. This implies that  $A_{\mathcal D}$ is precompact, therefore finite, since ${\mathcal D}$ is discrete.\\
	Let now $\lbrace a_1,\ldots,a_n\rbrace$ be a generating set of $A$. For every $i=1,\ldots,n$ we write $a_i = (s_i,  d_i, \tau_i) \in {\mathcal S} \times {\mathcal D} \times {\mathcal E}_n$.   
	The discussion above implies that there exists $K\in \mathbb{N}^*$ such that $a_i^K = (\text{id},  \text{id}, \tau_i^K)$ for all $i$, where  $\tau_i$ is a  translation of $\mathbb{R}^n$. Therefore the group
\vspace{-3mm}
	
	$$B=\lbrace (\text{id}, \text{id}, \tau) \in \Gamma \text{ s.t. } \tau \in \text{Transl}(\mathbb{R}^n) \rbrace$$
	is a normal, free abelian subgroup of rank $n$ of $\Gamma$. 
	It then follows by \cite[Theorem 2.(ii)]{CM19} that  $\Gamma$ has a normal, finite index subgroup $\check \Gamma$ which splits as $\check \Gamma = \check \Gamma_Y \times \mathbb{Z}^n$, where  $\check \Gamma_Y $ is a discrete group of  $Y := M \times N$, and $\mathbb{Z}^n$ acts as a lattice on $\mathbb{R}^n$. It is then  immediate   that also the projections  $\Gamma_Y$ and $\Gamma_{\mathbb{R}^n}$ of $\Gamma$ act discretely on the respective factors. 
	This proves (i) and (ii).\\
 Moreover, $Y$ is again a proper, geodesically complete, $(P_0,r_0)$-packed, CAT$(0)$-space and $\Gamma_Y$ acts $D_0$-cocompactly on Y.  Since $Y$ has no Euclidean factors,   Theorem \ref{theo-splitting-weak} implies that $\textup{sys}^\diamond(\Gamma_Y,Y) \geq \sigma_0^\ast =\sigma_{P_0,r_0,D_0}(\varepsilon_0)$,	 where $\varepsilon_0$ is the Margulis constant, which proves (iii). \\
 Finally, by Theorem \ref{theo-finiteness} the possible isomorphism types of nonsingular groups $\Gamma$ satisfying the assumptions of Theorem \ref{theo-group-splitting-Lie} are finite, so the index of $\check \Gamma$ in $\Gamma$ can be supposed uniformly bounded by a constant $I_0$ only depending on $P_0, r_0$ and $D_0$.  The last assertion is   consequence of (i) and of this bound. The fact that $\check{\Gamma}_Y$ is nonsingular if $\Gamma$ is nonsingular is obvious.  
\end{proof}


We remark that we do not know examples of  proper, geodesically complete, $\textup{CAT}(0)$-space $X$ whose isometry group is Lie, and  possessing a singular, discrete group $\Gamma < \textup{Isom}(X)$.

\section{Convergence and collapsing}
\label{sec-convergence}

We now move to the problem of convergence of  CAT$(0)$-group actions.\linebreak
Namely, the basic objects we are interested are  {\em isometry groups  of pointed spaces} (which we will restrict to uniform lattices of CAT$(0)$-spaces, from Section \ref{sub-standard} on),  that is    triples  
$$(\Gamma, X,x)$$ 
where $X$ is a proper metric space, $x \in X$ is a basepoint and $\Gamma$ is a {\em closed} (but not necessarily discrete, nor cocompact) subgroup of   $\text{Isom}(X)$.  \\
An {\em equivariant isometry} between isometric  actions of pointed spaces  $(\Gamma,X,x)$ and $(\Lambda, Y,y)$ is an isometry $F\colon X \to Y$ such that
\begin{itemize}[leftmargin=5mm]
	\item[--] $F(x)=y$;
	\item[--] $F_*\colon \textup{Isom}(X) \to \textup{Isom}(Y)$ defined by $F_*(g) = F\circ g \circ F^{-1}$ yields an isomorphism between $\Gamma$ and $\Lambda$.
\end{itemize}

\noindent We recall the definition of equivariant Gromov-Hausdorff convergence of isometric actions, as introduced by Fukaya  in  \cite[Ch.1]{Fuk86} and \cite[\S3]{Fuk-non-negative}.

	\begin{defin}
	\label{defin:equivariant_GH_approximation}
		Let $(\Gamma,X,x), (\Lambda, Y,y)$ be isometry groups of pointed spaces.\\
Given  $\varepsilon > 0$, an {\em equivariant $\varepsilon$-approximation} from $(\Gamma, X,x)$ to $(\Lambda,Y,y)$ is a triple $(f,\phi,\psi)$ where:
		\begin{itemize}[leftmargin=5mm]
			\item $f \colon B(x, \frac{1}{\varepsilon}) \to B(y,\frac{1}{\varepsilon})$ is a map such that $f(x)=y$ and satisfying
			\begin{itemize}[leftmargin=5mm]
				\item[-] $\vert d(f(x_1), f(x_2)) - d(x_1, x_2)\vert <\varepsilon$   $ \forall x_1,x_2\in B(x,\frac{1}{\varepsilon})$;
				\item[-]  $\forall y_1\in B(y, \frac{1}{\varepsilon})$   $\exists x_1\in B(x,\frac{1}{\varepsilon})$ such that $d(f(x_1),y_1) < \varepsilon$;
			\end{itemize} \vspace{1mm}
			\item $\phi\colon \Sigma_{\frac{1}{\varepsilon}}(\Gamma, x) \to \Sigma_{\frac{1}{\varepsilon}}(\Lambda, y)$ is a map satisfying $d(f(gx_1), \phi(g) f(x_1)) <\varepsilon$ \\  $ \forall g\in \Sigma_{\frac{1}{\varepsilon}}(\Gamma, x)$ and   $ \forall x_1\in B(x, \frac{1}{\varepsilon})$ such that  $gx_1 \in B(x, \frac{1}{\varepsilon})$;
			\item $\psi\colon \Sigma_{\frac{1}{\varepsilon}}(\Lambda, y) \to \Sigma_{\frac{1}{\varepsilon}}(\Gamma, x)$ is a map satisfying $d(f(\psi(g)x_1), g f(x_1)) <\varepsilon$ \\  $\forall g\in \Sigma_{\frac{1}{\varepsilon}}(\Lambda, y)$ and   $\forall x_1\in B(x, \frac{1}{\varepsilon})$ such that $\psi(g)x_1 \in B(x, \frac{1}{\varepsilon})$.
		\end{itemize}
	\end{defin}

	\begin{defin}
	\label{defin:equivariant_GH_distance}
		The {\em equivariant,  pointed Gromov-Hausdorff distance} between two isometry groups of pointed spaces $(\Gamma, X,x)$ and $( \Lambda, Y,y)$
		$$d_{\textup{eq-pGH}}  ( (\Gamma, X,x), ( \Lambda, Y,y)  )$$ 
		is the infimum of $\varepsilon $ for which there exists an equivariant $\varepsilon$-approximation from $(\Gamma,X,x)$ to  $(\Lambda,Y,y)$ and an equivariant $\varepsilon$-approximation from $(\Lambda,Y,y)$ to $(\Gamma,X,x)$. 
		\footnote{The definition is slightly different from the one in \cite{Fuk86} and it might not define a true distance. However the two definitions are comparable and in particular they define the same convergence.} 
		A sequence of isometry groups of pointed spaces $(\Gamma_n,X_n,x_n)$ {\em converges in the equivariant, pointed Gromov-Hausdorff  sense} to  $(\Gamma,X,x)$ if for every $\varepsilon > 0$ there exists $n_\varepsilon \geq 0$ such that if $n\geq n_\varepsilon$ then 
		$$d_{\textup{eq-pGH}}( (\Gamma_n,X_n,x_n), (\Gamma, X,x) ) < \varepsilon.$$
		In this case we write $(\Gamma_n,X_n,x_n ) \underset{\text{eq-pGH}}{\longrightarrow} (\Gamma, X,x)$.
	\end{defin}

\noindent For trivial groups  $\Gamma,  \Lambda, \Gamma_n,  \Lambda_n $  we recover the usual definition of pointed, Gromov-Hausdorff distance  and convergence (denoted $d_{\textup{pGH}}$ and  $\underset{\text{pGH}}{\longrightarrow}$); \linebreak 
moreover, for compact spaces with uniformily bounded diameter this reduces to the usual  Gromov-Hausdorff distance/convergence  $d_{\textup{GH}}$  (in the following, sometimes abbrev. {\em \textup{GH}-distance}).
\vspace{2mm}


In the sequel we will follow an equivalent approach which uses ultralimits, which is more adapted to our purposes.  
We will present it in the next subsection, and then recall the relations with  the usual notion of Gromov-Hausdorff convergence, referring to \cite{Jan17} and \cite{Cav21ter}; see also  \cite{DK18}, \cite{CavS20} for more details on ultralimits. \\ On the other hand, we will state our results in terms of Gromov-Hausdorff convergence, which is more commonly used in literature.


\subsection{Equivariant Gromov-Hausdorff convergence and ultralimits}
\label{sub-GH} ${}$\\  
A {\em non-principal ultrafilter} $\omega$ is a finitely additive measure on $\mathbb{N}$ such that $\omega(A) \in \lbrace 0,1 \rbrace$ for every $A\subseteq \mathbb{N}$ and $\omega(A)=0$ for every finite subset of $\mathbb{N}$. \linebreak We will write  {\em $\omega$-a.s.}  or {\em for $\omega$-a.e.$(j)$} in the usual measure theoretic sense.  
Given a bounded sequence $(a_j)$ of real numbers and a non-principal ultrafilter $\omega$ there exists a unique $a\in \mathbb{R}$ such that  the set $\lbrace j \in \mathbb{N} \text{ s.t. } \vert a_j - a \vert < \varepsilon\rbrace$ has $\omega$-measure $1$ for every $\varepsilon > 0$ (cp. \cite[Lemma 10.25]{DK18}).  The real number $a$ is then called {\em the ultralimit of the sequence $a_j$} and  is denoted by $\omega$-$\lim a_j$.
\vspace{1mm}

A     non-principal ultrafilter  $\omega$ being fixed,  for any  sequence of  pointed metric spaces  $(X_j, x_j)$ one can define   the {\em ultralimit pointed metric space} 
$$(X_\omega, x_\omega) = \omega\text{-}\lim (X_j, x_j)$$ 
-- first, one  says that a sequence    $(y_j)$, where $y_j\in X_j$ for every $j$, is {\em admissible} if there exists $M$ such that $d(x_j,y_j)\leq M$ for $\omega$-a.e.$(j)$; \\ 
-- then, one 
defines $(X_\omega, x_\omega)$ as  set of admissible sequences $(y_j)$ modulo the relation $(y_j)\sim (y_j')$ if and only if $\omega$-$\lim d(y_j,y_j') = 0$. \\
The point of $X_\omega$ defined by the class of the sequence $(y_j)$ is denoted by  $y_\omega = \omega$-$\lim y_j$. 
Finally, the formula 
$$d(\omega\text{-}\lim y_j, \omega\text{-}\lim y_j') := \omega\text{-}\lim d(y_j,y_j')$$
 defines a metric on $X_\omega$ which is called the {\em ultralimit distance} on $X_\omega$.
\vspace{1mm}

\noindent   Using a non-principal ultrafilter  $\omega$, one can also talk of limits of  isometries  and   of  isometry groups of pointed metric spaces. A sequence of isometries  $g_j $ of pointed metric spaces $(X_j, x_j)$ is {\em admissible} if there exists $M\geq 0$ such that $d(g_j x_j, x_j) \leq M$ $\omega$-a.s.. Any such sequence defines a limit isometry \linebreak 
  $g_\omega = \omega$-$\lim g_j$ of $X_\omega$  by the formula (see \cite[Lemma 10.48]{DK18}):  
$$g_\omega y_\omega := \omega\text{-}\lim g_jy_j.$$ 
Finally, given  sequence of isometry groups  of pointed spaces $(\Gamma_j, X_j,x_j)$ we can define the ultralimit group $\Gamma_\omega = \omega\text{-}\lim \Gamma_j$
as $$\Gamma_\omega = \lbrace \omega\text{-}\lim g_j \text{ s.t. } g_j \in \Gamma_j \text{ for } \omega\text{-a.e.}(j)\rbrace$$

\noindent In particular the elements of $\Gamma_\omega$ are ultralimits of admissible sequences. 
\vspace{-1mm}

\begin{lemma}[\cite{Cav21ter}, Lemma 3.7]
	\label{composition}
	The composition of admissible sequences of isometries is an admissible sequence of isometries and the limit of the composition is the composition of the limits.
\end{lemma}

\vspace{-2mm}
\noindent Therefore,  one has a well-defined composition law on $\Gamma_\omega$: if $ g_\omega   = \omega$-$\lim g_j$ and  $ h_\omega = \omega$-$\lim h_j$ we set $g_\omega \, \circ\, h_\omega := \omega\text{-}\lim(g_j \circ h_j).$
With this operation $\Gamma_\omega$ becomes a true group of isometries of $X_\omega$.\\
Notice that $\Gamma_\omega$ depends, in general,  on the chosen basepoints $x_j \in X_j$ (since this choice selects the admissible isometries), and that if $X_\omega$ is proper then $\Gamma_\omega$ is always a {\em closed} subgroup of isometries of $X_\omega$ \cite[Proposition 3.8]{Cav21ter}.

\vspace{1mm}
\noindent In conclusion,  a non-principal ultrafilter  $\omega$ being given, for any sequence of isometry groups of pointed  spaces $(\Gamma_j, X_j,x_j)$ there exists an {\em ultralimit isometry group of a pointed space}
\vspace{-3mm}

 $$(\Gamma_\omega, X_\omega, x_\omega) =  \omega \text{-} \lim (\Gamma_j, X_j,x_j) .$$ 

\vspace{2mm}
We now turn to the equivalence between the ultralimit approach and the best known notion of Gromov-Hausdorff convergence. 
 The relation between the  Gromov-Hausdorff convergence (simple, pointed,  and equivariant) and the convergence via ultralimits theory is resumed by the following result:

\begin{prop}[\textup{\cite[Proposition 3.13 \& Corollary 3.14]{Cav21ter}}] 
	\label{prop-GH-ultralimit} ${}$\\
Let   $(\Gamma_j, X_j,x_j)$ be a sequence of isometry groups of pointed spaces:
 \begin{itemize}[leftmargin=6mm] 
\item[(i)]   if $(\Gamma_j, X_j,x_j)  \hspace{-1mm}\underset{\textup{eq-pGH}}{\longrightarrow}  \hspace{-1mm}(\Gamma_\infty,X_\infty,x_\infty)$, then $(\Gamma_\omega, X_\omega, x_\omega) \cong (\Gamma_\infty,X_\infty,x_\infty)$  for every non-principal ultrafilter $\omega$;
\item[(ii)] reciprocally, if $\omega$  is a non-principal ultrafilter and the ultralimit  $(X_\omega, x_\omega)$ is proper  then    $(\Gamma_{j_k},X_{j_k},x_{j_k}) \hspace{-2mm} \underset{\textup{eq-pGH}}{\longrightarrow}  \hspace{-2mm} (\Gamma_\omega, X_\omega, x_\omega)$ for some subsequence $\lbrace{j_k}\rbrace$. \linebreak
Moreover,  if for every non-principal ultrafilter $\omega$ the ultralimit  $(\Gamma_\omega, X_\omega, x_\omega)$ is equivariantly isometric to the same isometry group  of pointed space   $(\Gamma, X,x)$, with $X$ proper,  then $(\Gamma_j, X_j,x_j) \hspace{-1mm} \underset{\textup{eq-pGH}}{\longrightarrow}  \hspace{-1mm}(\Gamma, X,x)$.
	\end{itemize}
{\em  When restricted to the case  $\Gamma_j =(1)$  for all $j$ (resp. when the spaces $X_j$ are compact),  this statement also explains  the standard relations between pointed Gromov-Hausdorff convergence (resp. simple Gromov-Hausdorff convergence) and ultralimit convergence, as studied for instance in \cite{Jan17}.	\\
In particular, this shows that if $(X_j, x_j)  \hspace{-1mm} \underset{\textup{pGH}}{\longrightarrow}  \hspace{-1mm} (X_\infty,x_\infty)$, then {\em any} family of  isometry groups $\Gamma_j < \textup{Isom}(X_j)$ (sub-)converges to some limit isometry group $\Gamma_\infty$ of $X_\infty$  with respect to the pointed, equivariant GH-convergence (since, for any choice of $\omega$,  the ultralimit $X_\omega=X_\infty$ is proper by (i), and then  there exists $(j_k)_{k \in \mathbb{N}}$ such that  $(\Gamma_{j_k}, X_{j_k},x_{j_k}) \hspace{-1mm} \underset{\textup{eq-pGH}}{\longrightarrow}  \hspace{-1mm} (\Gamma_\omega, X_\omega,x_\omega)$ by (ii)).}
\end{prop}


 A sequence $( \Gamma_j, X_j,x_j)$   is called {\em $D$-cocompact} if  each $\Gamma_j$ is $D$-cocompact, and {\em uniformly cocompact} if it is $D$-cocompact for some $D$.
It is not difficult to show that the ultralimit of a sequence of isometry groups of pointed spaces does not depend on the choice of the basepoints, 
 if the isometry groups are uniformly cocompact: 

\begin{lemma}
	\label{lemma-ultralimit-cocompact}
	Let $(\Gamma_j, X_j,x_j)$ be a sequence of isometry groups of pointed spaces,  all $D$-cocompact, and let  $(x_j')_{j \in {\mathbb N}}$ be a different sequence of basepoints.  \linebreak 
Then, for every   non-principal ultrafilter $\omega$, the ultralimit of   $(\Gamma_j, X_j,x_j)$ is equivariantly isometric to the ultralimit of  $(\Gamma_j,X_j,x_j')$.
\end{lemma}

\noindent Therefore, when considering the convergence of uniformily cocompact isometric actions, 
we will often omit the basepoint, if unnecessary for our arguments.\\
Finally, we remark that the equivariant pointed Gromov-Hausdorff convergence of a sequence of  isometric  actions   with uniformly bounded codiameter implies the pointed Gromov-Hausdorff convergence of the quotients.

\begin{lemma}[\textup{\cite[Theorem 2.1]{Fuk86}}] 
	\label{lemma-ultralimit-quotient}${}$\\
	Let $(\Gamma_j,X_j)$ be a sequence of $D$-cocompact isometry groups of pointed spaces. \\
	If  $(\Gamma_j,X_j) \underset{\textup{eq-pGH}}{\longrightarrow} (\Gamma_\infty,X_\infty)$ then $\Gamma_j\backslash X_j =: M_j \underset{\textup{GH}}{\longrightarrow} M_\infty := \Gamma_\infty \backslash X_\infty$.


\end{lemma}
\noindent Of course, even if the groups $\Gamma_j$ are all discrete, the group $\Gamma_\infty$ may be not discrete and the structure of the quotient $\Gamma_\infty \backslash X_\infty$ is not clear at all. 

\vspace{2mm}
\subsection{
Convergence of CAT$(0)$-lattices} 
\label{sub-standard}${}$

\noindent Our next goal is to explain what happens in our specific setting, where the $X_j$ are all proper, geodesically complete CAT$(0)$-spaces, and the $\Gamma_j$ are  uniformly cocompact.
With this purpose, we now precisely define the following classes of isometry groups  we are interested in:
\begin{itemize}[leftmargin=4mm] 
		\item[-] $\textup{CAT}_0(D_0)$: this is the class  
of isometry groups of spaces $(\Gamma,X)$
		where $X$ is a  proper, geodesically complete, CAT$(0)$-space,  and $\Gamma$ is a $D_0$-cocompact, lattice of $X$;
		\item[-] $\textup{CAT}_0(P_0,r_0,D_0)$:  the subclass of   $\textup{CAT}_0(D_0)$ made of the lattices $(\Gamma,X)$
such that  $X$ is, moreover,   $(P_0,r_0)$-packed.
\end{itemize}

\noindent We will denote by ${\mathcal{O}} \text{-CAT}_0 (D_0)$,  ${\mathcal O} \text{-CAT}_0 (P_0, r_0, D_0)$ the  respective classes of quotients $M=\Gamma \backslash X$: these are  compact CAT$(0)$-orbispaces, with 
$(\Gamma,X)$
respectively in $\text{CAT}_0 (D_0)$ and  $\text{CAT}_0 (P_0, r_0, D_0)$.  
\vspace{2mm}

  We will say that a sequence of proper metric spaces  $X_j$
(and, by extension, a  sequence of isometric actions  $(\Gamma_j,X_j, x_j)$) is {\em uniformly packed} if there exists $(P_0, r_0)$ such that every  $X_j$ is  $(P_0, r_0)$-packed. 
 The proof of \cite[Lemma 5.4]{Cav21ter} implies that
	$$\textup{CAT}_0(D_0) = \bigcup_{P_0, r_0} \textup{CAT}_0(P_0,r_0,D_0).$$
More precisely, we have: 

\begin{prop}
	\label{lemma-GH-compactness-packing}
	A subset $\mathcal{F} \subseteq \textup{CAT}_0(D_0)$ is precompact (with respect to the equivariant pointed Gromov-Hausdorff convergence) if and only if there exist $P_0,r_0 > 0$ such that $\mathcal{F}\subseteq \textup{CAT}_0(P_0,r_0,D_0)$.
\end{prop}

\begin{proof}
The ``if'' direction follows from the fact that 
the class of proper, geodesically complete CAT$(0)$-spaces  which are $(P_0, r_0)$-packed is precompact\footnote{and even compact, by Corollary 6.7 of \cite{CavS20}, since the dimension of any such space is bounded by $n_0=P_0/2$}, by \cite[Theorem 6.4]{CavS20}: that is,  every sequence of pointed spaces $(X_j,x_j)$ in this class (sub)-convergences to some proper metric space $(X_\infty,x_\infty)$ (which is still geodesically complete, CAT$(0)$ and $(P_0, r_0)$-packed). The conclusion then follows from the discussion after Proposition \ref{prop-GH-ultralimit}, showing that then any sequence of isometry groups $\Gamma_j$ of the pointed spaces $(X_j,x_j)$ subconverges in the pointed, equivariant GH-convergence to some isometry group $\Gamma_\infty$ of $(X_\infty,x_\infty)$.\\
The other direction is provided by \cite[Lemma 5.8]{Cav21ter}, which shows that if  $(\Gamma_j,X_j)$ is a sequence of $D_0$-cocompact isometry groups of spaces which converges in the equivariant pointed Gromov-Hausdorff sense to  some isometry group $ (\Gamma_\infty,X_\infty)$,  then the sequence $(\Gamma_j,X_j)$ is uniformly packed.  
\end{proof}

 Let us now consider a sequence of lattices  
$(\Gamma_j,X_j)$ in $\textup{CAT}_0(D_0)$, such that 
$(\Gamma_j,X_j) \underset{\textup{eq-pGH}}{\longrightarrow} (\Gamma_\infty,X_\infty)$.
We will always denote by $M_j = \Gamma_j\backslash X_j$ and $M_\infty = \Gamma_\infty \backslash X_\infty$ the quotient spaces.
 Our goal is to describe the limit group $\Gamma_\infty$ acting on $X_\infty$  and the quotient $M_\infty= \Gamma_\infty \backslash X_\infty$.

\begin{defin}[Standard setting of convergence]\label{defsetting}${}$\\
We say that we are in the \emph{standard setting of convergence} when we have a sequence of lattices $(\Gamma_j,X_j)$ in $ \textup{CAT}_0(D_0)$ such that 
$(\Gamma_j,X_j) \underset{\textup{eq-pGH}}{\longrightarrow} (\Gamma_\infty,X_\infty)$. 
Notice that, in this setting:
\begin{itemize}[leftmargin=8mm]
\item[(i)] we can always assume that the $(\Gamma_j,X_j)$'s belong to some CAT$_0(P_0,r_0,D_0)$, for some constants $P_0,r_0$ (by Proposition \ref{lemma-GH-compactness-packing});
\item[(ii)] the limit space $X_\infty$ is a proper, geodesically complete $\textup{CAT}(0)$-space which is still $(P_0, r_0)$-packed (since this class of spaces is closed under  ultralimits and compact with respect to the pointed Gromov-Hausdorff convergence, by  \cite[Theorem 6.1]{CavS20});
\item[(iii)] the limit group  $\Gamma_\infty$ may well be non-discrete, but it is always a closed subgroup of Isom$(X_\infty)$ (as recalled after Lemma \ref{composition}, by \cite{Cav21ter}), \linebreak in particular $M_\infty$ is always a genuine metric space;
\item[(iv)] the quotients $M_j =\Gamma_j \backslash X_j \underset{\textup{GH}}{\longrightarrow} M_\infty =\Gamma_\infty \backslash X_\infty$ (by Lemma \ref{lemma-ultralimit-quotient}${}$).
\end{itemize}

\noindent A note of attention should be put on (iv): the convergence of the quotient orbispaces $M_j \underset{\textup{GH}}{\longrightarrow}  M_\infty$ does not imply in general
\footnote{This is false even for sequences of Riemannian orbifolds with uniformily bounded sectional curvature, contrarily to what asserted in \cite[Proposition 2.7]{Fuk86} 
(where the non-collapsing assumption is forgotten). 
For instance, the product of a closed Riemannian manifold $M=\Gamma \backslash X$ with a torus $T_j = \Gamma_j \backslash {\mathbb R}^2$ converges  to $M$ by taking  $\Gamma_j =\varepsilon_j {\mathbb Z}^2$ with $\varepsilon_j \rightarrow 0$, but $(\Gamma \times \Gamma_j, X \times {\mathbb R}^2)$ does not converge to $(\Gamma, X)$.}
the convergence of the isometry groups $(\Gamma_j,X_j) \underset{\textup{eq-pGH}}{\longrightarrow} (\Gamma_\infty,X_\infty)$ (as this is false even for constant sequences, by Remark \ref{obs-samequotient}).

\noindent Moreover, we will say, for short, that we are in the \emph{nonsingular} setting of convergence (resp. in the {\em Lie setting}) if each group $\Gamma_j$ is nonsingular (resp. all isometry groups Isom$(X_j)$ are Lie).\\
Finally, a sequence of lattices $(\Gamma_j,X_j) \underset{\textup{eq-pGH}}{\longrightarrow} (\Gamma_\infty,X_\infty)$ in the standard setting   of convergence (equivalently, the sequence of orbispaces $M_j \underset{\textup{GH}}{\longrightarrow}  M_\infty$) \linebreak will be  called:
	\begin{itemize}[leftmargin=5mm]
		\item[-]  \emph{non-collapsing}, if $\limsup_{j\to+\infty} \textup{sys}^\diamond(\Gamma_j,X_j) > 0$;  
		\item[-] \emph{collapsing}, if $\liminf_{j\to+\infty} \textup{sys}^\diamond(\Gamma_j,X_j) = 0$.
	\end{itemize}
\end{defin}

\noindent These cases are not apparently mutually exclusive, but indeed they are. \linebreak
Actually,  the two cases  are respectively equivalent to the the conditions that the dimension  of the limit  $M_\infty = \Gamma_\infty \backslash X_\infty$  equals  the dimension  of the quotients $M_j$ or decreases,  as we  will   prove   in Subsection \ref{sub-characterization} (Theorem \ref{theo-collapsing-characterization}).  This is a very intuitive but not trivial result, 
and
will follow by a careful analysis of both the collapsed and non-collapsed  case.

\vspace{2mm}
\subsection{Convergence without collapsing} 
\label{sub-convergencenocollapsing}
 \begin{theo}
	\label{theo-noncollapsed}
Assume  $(\Gamma_j,X_j) \underset{\textup{eq-pGH}}{\longrightarrow} (\Gamma_\infty,X_\infty)$ in the standard setting, \linebreak 
without collapsing. Then: 
 \begin{itemize}[leftmargin=9mm] 
\item[(i)] 	the limit group $\Gamma_\infty$ is closed and
 totally disconnected;
\item[(ii)]  moreover, if all the   the groups $\Gamma_j$ are nonsingular, then  $\Gamma_\infty$ is discrete and isomorphic to $\Gamma_j$ for $j$ big enough.
\end{itemize}
 Further,  the limit of the  orbispaces $M_j \! =\Gamma_j \backslash X_j$ is isometric to $M_\infty \! =\Gamma_\infty \backslash X_\infty$ \linebreak and, in case (ii), $M_j$ is equivariantly homotopy equivalent to $M_\infty$ 
for $j\gg 0$.
\end{theo}
 
 
\begin{proof}[Proof of Theorem \ref{theo-noncollapsed}]${}$\\
We choose   a subsequence  $\lbrace j_h \rbrace$ be for which 
$\lim_{h\to +\infty} \text{sys}^\diamond(\Gamma_{j_h},X_{j_h}) \! >0$ exists, \linebreak and we choose a non-principal ultrafilter $\omega$ with $\omega(\lbrace j_h \rbrace) = 1$ (which is always possible, see \cite[Lemma 3.2]{Jan17}). Therefore we   assume that
$$\textup{sys}^\diamond(\Gamma_j,X_j) \geq \varepsilon > 0 \mbox{ for }\omega\mbox{-a.e.} (j).$$ 
By Proposition \ref{prop-GH-ultralimit}, it is enough to show the thesis for the ultralimit group $\Gamma_\omega$ of $X_\omega$, which coincide respectively  with $\Gamma_\infty$ and $X_\infty$. \\
By the remark (i) after \ref{defsetting}, we may assume that all the $(\Gamma_j, X_j)$'s belong to the class CAT$_0(P_0,r_0,D_0)$, for suitable $P_0,r_0$, and  that $\varepsilon \leq \varepsilon_0$. \\
Let then $\sigma = \sigma_{P_0,r_0,D_0}(\frac{ \varepsilon}{4J_0})$ be the constant provided by Theorem \ref{theo-splitting-weak}.  \\
  By Lemma \ref{lemma-ultralimit-cocompact} the limit triple does not depend on the choice of the basepoints, so we may assume that the  basepoints $x_j$ are those provided by Corollary \ref{cor-0-rank}(ii), that is the almost stabilizers  $\overline{\Gamma}_{\sigma} (x_j) < \Gamma_j$ fix  the basepoint $x_j$ for $\omega$-a.e.$(j)$.	We have to show that $\Gamma_\omega^\circ = \lbrace \text{id}\rbrace$. \\
	It is clear that if $g_\omega = \omega$-$\lim g_j \in \overline{\Sigma}_{\frac{\sigma}{2}}(x_\omega)$ then $g_j \in \overline{\Sigma}_\sigma(x_j)$ for $\omega$-a.e.$(j)$. 
	Therefore, $g_jx_j = x_j$ by our choice of $x_j$, so $g_\omega x_\omega = x_\omega$. 
	Notice that the subset $\overline{\Sigma}_{\frac{\sigma}{2}}(x_\omega) \subset \Gamma_\omega$ has non-empty interior: in fact,  calling $\varphi_{x_\omega}: \Gamma_\omega \rightarrow X_\omega$ the  continuous map $\varphi_{x_\omega} (g_\omega) = g_\omega x_\omega$, 
	  the open set  $U = \varphi_{x_\omega}^{-1} (B(x_\omega,  {\frac{\sigma}{2}}))$ is contained in $\overline{\Sigma}_{\frac{\sigma}{2}}(x_\omega)$. Therefore $U \cap \Gamma^\circ_\omega$ is an open neighbourhood of the identity in $\Gamma^\circ_\omega$, and it generates  $\Gamma^\circ_\omega$; it then follows that $\langle \Sigma_{\frac{\sigma}{2}}(x_\omega) \rangle \cap \Gamma_\omega^\circ = \Gamma_\omega^\circ$. In particular every element of $\Gamma_\omega^\circ$ fixes $x_\omega$. The set $\text{Fix}(\Gamma_\omega^\circ)$ is convex, closed, non-empty and $\Gamma_\omega$-invariant because $\Gamma_\omega^\circ$ is a normal subgroup of $\Gamma_\omega$. The action of $\Gamma_\omega$ on $X_\omega$ is clearly $D_0$-cocompact, hence minimal, so $\text{Fix}(\Gamma_\omega^\circ) = X_\omega$ and $\Gamma_\omega^\circ = \lbrace \text{id} \rbrace$. 
The facts that $\Gamma_\omega$ is closed and that $M_\infty$ is the limit of the $M_j$'s are just the remarks (iii)\&(iv) after Definition \ref{defsetting}.\\
To prove part (ii) first recall that by Theorem \ref{prop-vol-sys-dias}, up to replacing $\varepsilon$ with a smaller number (and possibly changing again the basepoint by Lemma \ref{lemma-ultralimit-cocompact}), we also have $\text{sys}(\Gamma_j,x_j) \geq  \varepsilon  > 0$ for $\omega$-a.e.$(j)$. 
It is then easy to deduce that $\textup{sys}(\Gamma_\omega,x_\omega)\geq \varepsilon$ too
(actually, if $g_\omega =  \omega$-$\lim g_j$ is any element of $\Gamma_\omega$, then $d(x_j, g_j x_j) \geq \varepsilon$ for $\omega$-a.e.$(j)$, hence $d(x_\omega, g_\omega x_\omega) \geq  \varepsilon$).
Since $X_\omega$ is proper, this is enough to show that $\Gamma_\omega$ is discrete, 
because the orbit $\Gamma_\omega x_\omega$ is discrete and the stabilizer of $x_\omega$
is trivial.   \\
Finally, choose $D_0 < D < D_0+1$ such that $d(x_\omega, g_\omega x_\omega) \neq 2D$ for all $g_\omega \in \Gamma_\omega$, and 
consider the generating sets $\Sigma_j := \overline{\Sigma}_{2D}(x_j)$ and $\Sigma_\omega := \overline{\Sigma}_{2D}(x_\omega)$  for $\Gamma_j$ and $\Gamma_\omega$, respectively. It is classical that the spaces $X_j$ and $X_\omega$ are respectively quasi-isometric to the Cayley graphs of $\text{Cay}(\Gamma_j, \Sigma_j)$ and $\text{Cay}(\Gamma_\omega, \Sigma_\omega)$  with   the respective word metrics    $d_{\Sigma_j}$ and $d_{\Sigma_\omega}$, 
namely
$$ \frac{1}{2D} \cdot d (g x_j, g' x_j)  \leq  d_{\Sigma_j} (g,g') \leq \frac{1}{2(D-D_0)}d (g x_j, g' x_j) +1 $$
and analogously for $d_{\Sigma_\omega}$ and $X_\omega$. From this, one can prove exactly as    in \cite[Proposition 7.6]{Cav21ter}  that  
$$\big(  \text{Cay}(\Gamma_\omega, \Sigma_\omega),   \lbrace \text{id}\rbrace \big)   
= \omega \text{-} \lim \big(  \text{Cay}(\Gamma_j, \Sigma_j),   \lbrace \text{id}\rbrace    \big).$$
But, since by Theorem  \ref{theo-finiteness} we  only have   finitely many  marked groups $(\Gamma_j,\Sigma_j)$, this
clearly implies that the $(\Gamma_j,\Sigma_j)$'s must be all isomorphic for $j$ big enough, in particular isomorphic to $(\Gamma_\omega, \Sigma_\omega)$.\\
\noindent The second to last assertion follows from   the fact that the class of proper, geodesically complete,   $(P_0, r_0)$-packed, pointed  $\textup{CAT}(0)$-spaces $(X_j,x_j)$ is closed under  ultralimits (cp. \cite[Theorem 6.1]{CavS20}) and from Lemma \ref{lemma-ultralimit-quotient}; clearly, the diameter of the limit     $\Gamma_\omega \backslash X_\omega$ is still bounded by $D_0$.\\
Finally, by \cite[Theorem 5.4]{CS23}, any isomorphism $\varphi_j: \Gamma_j \rightarrow \Gamma_\omega$ between nonsingular  uniform CAT$(0)$-lattices can be upgraded to $\varphi$-equivariant homotopy equivalence $f_j: X_j \rightarrow X_\omega$, which proves that $M_j$ and $M_\omega$ are equivariantly homotopy equivalent.
\end{proof}

\subsection{Convergence with  collapsing}
\label{sub-convergencecollapsing} 

\begin{theo}
	\label{theo-collapsed} 
Assume  $(\Gamma_j,X_j) \underset{\textup{eq-pGH}}{\longrightarrow} (\Gamma_\infty,X_\infty)$ in the standard setting, \linebreak
with collapsing, and let $\Gamma_\infty^\circ$ be the identity component of $\Gamma_\infty$. Then: 
	\begin{itemize}[leftmargin=7mm] 
		\item[(i)]  $X_\infty$ splits isometrically and $\Gamma_\infty$-invariantly as $ X_\infty' \times \mathbb{R}^\ell$, for some $ \ell \geq 1$; 
		\item[(ii)]  $\Gamma_\infty^\circ  =  \lbrace \textup{id}\rbrace \times \textup{Transl}(\mathbb{R}^\ell)$, and 
 $ X_\infty' =\Gamma_\infty^\circ \backslash X_\infty$.		 		
	\end{itemize}
The space $ X_\infty' $ is  proper, geodesically complete, $(P_0,r_0)$-packed and $\textup{CAT}(0)$.\\	 
 The orbispaces $M_j=\Gamma_j \backslash X_j$  converge to $M_\infty = \Gamma_\infty \backslash X_\infty$, which  is isometric to the quotient of $ X_\infty'$ by the closed,  totally disconnected group $   \Gamma_\infty' := \Gamma_\infty /\Gamma_\infty^\circ $.

\end{theo}

Recall that all the spaces under consideration, being $(P_0, r_0)$ packed, have dimension   $\leq n_0=P_0/2$, by  Proposition \ref{prop-packing}. Also,  recall the constant $J_0$ introduced in Section \ref{subsec-splitting}, which bounds the index of the canonical lattice ${\mathcal L}(G)$ in any crystallographic group $G$  of $\mathbb{R}^k$, for $k\leq n_0$. \\ To prove Theorem \ref{theo-collapsed}, we will use the following general fact, whose proof is straightforward and left to the reader:
 

\begin{lemma}
	\label{lemma-ultralimit-product}
	Let $(\Gamma_j,X_j,x_j), (\Gamma_j', X_j',x_j',)$ be two sequences as above,  $\omega$ be a non-principal ultrafilter.   Then the ultralimit of   
$(\Gamma_j \times \Gamma_j', X_j\times X_j', (x_j,x_j'))$ is equivariantly isometric to 
$(\Gamma_\omega \times \Gamma_\omega', X_\omega \times X_\omega', (x_\omega, x_\omega'))$.
\end{lemma}

We will also need the following, whose proof is included for completeness:
\begin{lemma}
	\label{lemma-translations}
	Let $A$ be a group of translations of  $\mathbb{R}^k$.
	Then $A \cong  \mathbb{R}^{\ell} \times \mathbb{Z}^{d}  $ \linebreak with $\ell + d \leq k$. Moreover, there is a corresponding $A$-invariant metric factorization of $\mathbb{R}^k$ as $\mathbb{R}^{\ell} \times \mathbb{R}^{k-\ell} $, such that the connected component $A^\circ \cong \mathbb{R}^{\ell}$ can be identified with $ \textup{Transl}(\mathbb{R}^\ell)$.
\end{lemma}

\begin{proof}
	Let  $\tau_{ v}$ denote the translation of $ \mathbb{R}^k$ by a vector $ v$. 
	Consider the connected component  $A^\circ$  of the identity of $A$: it is an abelian, connected group of translations. The group  $A^\circ$ acts transitively on the the vector subspace 
	$V=\lbrace v\in \mathbb{R}^k \text{ s.t. } \tau_{v} \in A^\circ\rbrace$; moreover, as $A^\circ$ acts by translations, the stabilizer of any point $ v$ is trivial, therefore  $A^\circ$ can be identified with $V \cong \mathbb{R}^{\ell} $. \linebreak
	We can split    $\mathbb{R}^k$   isometrically  as $V \oplus V^\perp $, where $V^\perp$ is the orthogonal complement of $V$. 
	Since $A$ acts by translations, each isometry $a\in A$ can be decomposed as $(a_1,a_2)$, with $a_1 \in V$ and  $a_2 \in V^\perp$.
	Notice that the isometry $(\text{id},a_2) \in A$ because the action of $A^\circ$ on $V$ is transitive;  so, if $(a_1,a_2) \in A$, also the isometry $(a_1,\text{id})$ belongs to $A$. In other words, $A =  A^\circ  \times  A^\perp $, where   $A^\perp$ is the projection of $A$ on $\text{Isom}(V^\perp)$,  
	Notice that $A^\perp = A / A^\circ$  is totally disconnected, hence discrete and free abelian, so it is isomorphic to some $\mathbb{Z}^d$ with $d\leq k-\ell$, acting on the factor $\mathbb{R}^{k-\ell} $.
\end{proof}
\vspace{1mm}

\begin{proof}[Proof of Theorem \ref{theo-collapsed}]${}$\\
	As we did in the proof of Theorem \ref{theo-noncollapsed} we choose a non-principal ultrafilter $\omega$ such that $\omega$-$\lim\textup{sys}^\diamond(\Gamma_j,X_j) = 0$, and by Proposition \ref{prop-GH-ultralimit} it is enough to show the thesis for the ultralimit   $(\Gamma_\omega,X_\omega)=(\Gamma_\infty, X_\infty)$. \\ Again, by remark (i) after \ref{defsetting}, we may assume that all the $(\Gamma_j, X_j)$'s belong to CAT$_0(P_0,r_0,D_0)$, for some constants $P_0,r_0$. \\
By our choice of the ultrafilter $\omega$, we have  $\textup{sys}^\diamond(\Gamma_j,X_j) \leq \sigma_0^\ast$  for $\omega$-a.e.$(j)$, 
	where $\sigma_0^\ast=  \sigma_{P_0,r_0,D_0}(\varepsilon_0)$ is the constant provided by Theorem \ref{theo-splitting-weak} for $\varepsilon=\varepsilon_0$, the Margulis constant; so   $X_j$ splits as $Y_j\times \mathbb{R}^{k_j}$ with $k_j \geq 1$ for $\omega$-a.e.$(j)$.
By Lemma \ref{lemma-ultralimit-cocompact} the limit does not depend on the choice of the basepoints, so we can assume that the basepoints are the points $x_j=(y_j,  v_j)$ provided by Proposition \ref{prop-minimal-close}, that is such that  $\overline{\Gamma}_{ \sigma_0^\ast}(x_j)$ preserves the slice $\lbrace y_j \rbrace \times \mathbb{R}^{k_j}$.
Moreover, if $k_\omega = \omega$-$\lim k_j$, we have  $k_j = k_\omega$ for $\omega$-a.e.$(j)$ and $k_\omega\leq n_0$.\\
Now, choose a positive 
$\varepsilon 
\leq \frac12 \min \left\{ \sigma_0^\ast ,  
\sqrt{2 \sin \left( \frac{\pi}{J_0}\right) }\right\}$
Then, by Lemma \ref{lemma-bieber}, we deduce that   every element $g$ of a crystallographic group of $\mathbb{R}^k$, for $k\leq n_0$,  moving every point of $B_{\mathbb{R}^k} (O, \frac{1}{2\varepsilon})$ less than $2\varepsilon$ is a translation.
	We consider the open subset of $\Gamma_\omega$ defined by
	$$U(x_\omega, \varepsilon) :=  \big\{ g_\omega \in \Gamma_\omega \text{ s.t. } d(y_\omega, g_\omega y_\omega) < \varepsilon \text{ for all } y_\omega \in B_{X_\omega} (x_\omega, 1/\varepsilon  ) \big\}.$$
	  Since $\langle U(x_\omega, \varepsilon) \rangle$ is an  open subgroup  of $\Gamma_\omega$ containing the identity, we have $\langle U(x_\omega, \varepsilon) \rangle\cap \Gamma_\omega^\circ = \Gamma_\omega^\circ$, for every $\varepsilon >0$. 
Remark that if 	$g_\omega = \omega$-$\lim g_j $ is in $ U(x_\omega, \varepsilon)$,  then 
for $\omega$-a.e.$(j)$ the isometry $g_j$ belongs to 
	$$U(x_j,2\varepsilon) := \big\{  g_j \in \Gamma_j \text{ s.t. } d(y_j, g_j y_j) < 2\varepsilon \text{ for all } y_j \in B_{X_j}\left(x_j, 1/(2\varepsilon) \right)  \big\}.$$
Every element $g_j \in U(x_j,2\varepsilon)$ belongs to  $\overline{\Gamma}_{ \sigma_0^\ast}(x_j)$ since $2\varepsilon \leq \sigma_0^\ast$, therefore it acts on $X_j= Y_j \times \mathbb{R}^{k_\omega}$ as  $g_j = (g_j',g_j'') $, where   $g'_j$ fixes $y_j$ (because of our choice of basepoints); moreover,   $g_j''$ is  a global translation of $\mathbb{R}^{k_\omega}$, since it moves the points of   $B_{\mathbb{R}^{k_\omega}} (O, \frac{1}{2\varepsilon})$ less than $2\varepsilon$. 
An application of Lemma \ref{lemma-ultralimit-product}	says that also $X_\omega$ splits isometrically as $Y_\omega \times \mathbb{R}^{k_\omega}$, where $Y_\omega$ is the ultralimit of the  $(Y_j,y_j)$'s. Moreover clearly $\Gamma_\omega$ preserves the product decomposition.
 In particular every element of $g_\omega \in \Gamma_\omega^\circ$ can be written as a  product $u_\omega(1) \cdots u_\omega(n)$, with each $u_\omega(k) $ in $ U(x_\omega, \varepsilon)$; 
 as $u_\omega(k) = \omega$-$\lim  u_j(k)$ with $u_j(k)= (u_j(k)', u_j(k)'') \in U(x_j, 2\varepsilon)$, where  $u_j(k)'$ fixes $y_j$ and $u_j(k)''$ is a translation, it follows that also
  $g_\omega$ can be written as $(g_\omega',g_\omega'') \in \text{Isom}(Y_\omega) \times \text{Isom}(\mathbb{R}^{k_\omega})$ where  $g_\omega' y_\omega = y_\omega$  and $g_\omega''$ is a global translation   (being the ultralimit of Euclidean translations). \\
  Let us call $p\colon \Gamma_\omega \to \text{Isom}(Y_\omega)$ the projection map. The group $p(\Gamma_\omega^\circ)$ is normal in $p(\Gamma_\omega)$. The set of fixed points $\text{Fix}(p(\Gamma_\omega^\circ))$ is closed, convex, non-empty and $p(\Gamma_\omega)$-invariant (since $p(\Gamma_\omega^\circ)$ is normal in $p(\Gamma_\omega ))$. Since $\Gamma_\omega$ is clearly $D_0$-cocompact, so it is $p(\Gamma_\omega)$. Then the  action of $p(\Gamma_\omega)$ on $Y_\omega$ is minimal (cp. \cite[Lemma 3.13]{CM09b}) which implies that $\text{Fix}(p(\Gamma_\omega^\circ)) = Y_\omega$, that is $p(\Gamma_\omega^\circ) = \lbrace \text{id}\rbrace$. 
  We conclude that $\Gamma_\omega^\circ$ is a connected subgroup of 
    $\lbrace \text{id}\rbrace \times \text{Transl}(\mathbb{R}^{k_\omega})$. \\
     By Lemma \ref{lemma-translations},  $\mathbb{R}^{k_\omega}$ splits isometrically as $\mathbb{R}^\ell \times \mathbb{R}^{k_\omega-\ell}$ and $ \Gamma_\omega^\circ$ can be identified with the subgroup of translations of the factor  $\mathbb{R}^\ell$,
 for some $\ell \leq k_\omega$. 
Setting $X_\omega':=(Y_\omega \times \mathbb{R}^{k_\omega-\ell})$, this is still a  proper, geodesically complete,  $(P_0,r_0)$-packed, $\text{CAT}(0)$-space, and clearly $X_\omega' =\Gamma_\omega^\circ \backslash X_\omega$. 
Notice that, since $\Gamma_\omega^\circ$ is normal in $\Gamma_\omega$,  then  the splitting $X_\omega =X_\omega' \times \mathbb{R}^\ell$ is $\Gamma_\omega$-invariant.\\
Let us now  show that $ \Gamma_\omega^\circ$ is non-trivial, hence $ \ell \geq 1$.
Actually, 	if $ \Gamma_\omega^\circ$ was trivial then $\Gamma_\omega$ would be  totally disconnected, and 
by \cite[Corollary 3.3]{Cap09}, we would have $\text{sys}^\diamond(\Gamma_\omega, X_\omega) > 0$. However, we are able to contruct hyperbolic isometries of $\Gamma_\omega$ with arbitrarily small translation length, which will prove  that $\Gamma_\omega^\circ$ is non-trivial.  Indeed, fix any $\lambda > 0$ and an error $\delta > 0$. \linebreak
By the collapsing assumption we can find hyperbolic isometries $g_j \in \Gamma_j$ with $\ell(g_j) \leq \delta$, for $\omega$-a.e.$(j)$. By $D_0$-cocompactness, up to conjugating $g_j$   we can suppose  $g_j$ has an axis  at distance at most $D_0$ from $x_j$. Take a power $m_j$ of $g_j$ such that $\lambda < \ell(g_j^{m_j}) \leq \lambda + \delta$. Then, the sequence $(g_j^{m_j})$ is admissible and defines a hyperbolic element of $\Gamma_\omega$ whose translation length is between $\lambda$ and $\lambda + \delta$. By the arbitrariness of $\lambda$ and  $\delta$ we conclude.\\
Finally, the quotients  $M_j=\Gamma_j \backslash X_j$  converge to $M_\infty = \Gamma_\infty \backslash X_\infty = \Gamma_\omega \backslash X_\omega$, 	 by the remark (iv) after  \ref{defsetting}
and $\Gamma_\omega \backslash X_\omega = (\Gamma_\omega /\Gamma_\omega^\circ) \backslash (\Gamma_\omega^\circ \backslash X_\omega) =  \Gamma_\omega'   \backslash X_\omega'$.   \\
It is easy to see that the totally disconnected group $\Gamma_\omega / \Gamma_\omega^\circ$ is closed as a subgroup of \textup{Isom}$(X_\omega')$. 
\end{proof}

\begin{obs}\label{rem-totdiscvsdicrete} 
{\em 
\noindent In full generality, the conclusion of Theorem \ref{theo-collapsed} cannot be improved saying that the  totally disconnected group   $\Gamma_\infty / \Gamma_\infty^\circ$  is discrete. \\
Indeed, let $(X_j,  \Gamma_j)$ be the sequence considered in the Example \ref{ex-comm-not-normal}, where none of the groups $ \Gamma_j$ have non-trivial,  abelian, virtually normal subgroups. The sequence converges in the pointed equivariant Gromov-Hausdorff sense to a limit isometric action $(X_\infty,   \Gamma_\infty)$, where $X_\infty = \mathbb{R}^2 \times T$ and the connected component of $\Gamma_\infty$ is $\Gamma_\infty^\circ = \text{Transl}(\mathbb{R}^2) \times \lbrace \text{id}\rbrace$, so $X_\infty' = T$. The group $\Gamma_\infty / \Gamma_\infty^\circ$, acting on $T$, contains the infinite order, elliptic isometries induced by $a$ and $b$, so it is not discrete.
In particular $\Gamma_\infty$ is not a Lie group.\\
On the other hand  observe that, in general, if $\Gamma_\infty$ is a Lie group, then one can use the same proof of \cite[Lemma 58]{SZ22}   to produce a non-trivial, nilpotent, normal subgroup  of $\Gamma_j$, for $j$ big enough. This means that the existence of non-trivial, nilpotent, normal subgroups in $\Gamma_j$ for $j$ big enough, and the corresponding splitting of the groups $\Gamma_j$, is strictly related to the structure of the limit group.
}
\end{obs}
In Section \ref{sub-collapsinglie}${}$  we will see a more precise relation between the metric limit $M_\infty$ and the approximants $M_j$  in the  standard, nonsingular setting of convergence in the collapsing case, provided that the isometry  groups Isom$(X_j)$ are Lie.

\subsection{Riemannian limits}
\label{sec:riemannian}
${}$

\noindent Recall that any Riemannian, geodesically complete   CAT$(0)$-space $X$ is a {\em Hadamard manifold}, i.e. a complete simply connected manifold with nonpositive sectional curvature.  If the limit $M_\infty$ of a sequence of CAT$(0)$-orbispaces is a Riemannian manifold, then the equivariant homotopy equivalence between $M_j$ and $M_\infty$ of Theorem \ref{theo-noncollapsed} can be promoted to homeomorphism.

\noindent For this, we need the following preliminary fact:

\begin{theo}	
\label{prop-Riemannian-close-topological}
Let  $X$ be a $n$-dimensional Hadamard manifold. There exists $\varepsilon=\varepsilon(X, D_0)>0$ such that if a proper, geodesically complete, $D_0$-cocompact \textup{CAT}$(0)$-space $X'$ satisfies $$d_{\textup{p-GH}}((X,x),(X',x'))<\varepsilon$$ (for some choice of basepoints $x,x'$) then $X'$ is a topological manifold. 	
\end{theo}

	
\begin{proof}
Reasoning by contradiction, take a sequence of (proper, geodesically complete) pointed CAT$(0)$-spaces $(X_j,x_j)$  admitting $D_0$-cocompact lattices $\Gamma_j$, and converging in  pointed GH-distance to a Hadamard manifold $(X,x)$,    and suppose that the $X_j$'s  are  definitely not  topological manifolds.  \linebreak
Since $(X_j,x_j)    \underset{\textup{pGH}}{\longrightarrow} (X,x)$, then (up to a subsequence) we may assume that the lattices $(\Gamma_j, X_j)$  converge to some $D_0$-cocompact (possibly non-discrete) isometry group $\Gamma$ of $X$. Then, by Proposition \ref{lemma-GH-compactness-packing},    the   lattices  $(\Gamma_j, X_j) $ belong definitely to $\textup{CAT}_0(P_0,r_0,D_0)$, for some $P_0$ and $r_0$.\\
Moreover, by $D_0$-cocompactness, the distance of $x_j$ to the stratum $X^{n_j}_j$ of $X_j$  of  maximal dimension is uniformily bounded above; then, the  dimension of $X_j$ is exactly $n$ for $j \gg 0$, by \cite[Proposition 6.5]{CavS20}. \\
Now, using the terminology of \cite{LN19},  every point of $X$ is $(n,0)$-strained. \\
Then, we set $\delta := \frac{1}{50 n^2}$ and we claim that every point of $\overline{B}(x_j, 2D_0)$ is $(n,\delta)$-strained if $j$ is larger than some constant $ j_\delta$. Actually, suppose that we could find a sequence of points $(y_j)_{j \in {\mathbb N}}$, each one in $ \overline{B}(x_j, 2D_0)$,  not $(n,\delta)$-strained. \linebreak
This sequence defines a point $y  \in X$, which is $(n,0)$-strained and in particular also $(n,\delta)$-strained.  
Since all the balls $ \overline{B}(x_j, 2D_0)$ are uniformly packed, we can use   \cite[Lemma 7.8]{LN19}  to deduce that also the points $y_j$ would be $(n,\delta)$-strained, a contradiction. 
So, the fact that all points of $\overline{B}(x_j, 2D_0)$ are $(n,\delta)$-strained implies that $B(x_j, 2D_0)$ is locally biLipschitz homeomorphic to $\mathbb{R}^n$  by \cite[Corollary 11.8]{LN19}. In particular, $B(x_j, 2D_0)$ is a topological manifold, and by $D_0$-cocompactness we conclude that the whole $X_j$ are topological manifolds, which gives a contradiction.  		 
	\end{proof}


\begin{cor}\label{cor-homeo}
 Let $(\Gamma,X)$ be a torsionless, uniform lattice of a Hadamard manifold.
 Then, there exists
 $\varepsilon=\varepsilon(\Gamma, X)\!>0$
such that the following holds: for every other \textup{CAT}$(0)$-lattice  $(\Gamma', X')$, if  
 $d_{\textup{eq-pGH}}((\Gamma,\! X),(\Gamma',\! X'))<\!\varepsilon$ 
then $M'=\Gamma' \backslash X'$ is homeomorphic to $M=\Gamma\backslash X$.
\end{cor}

 \begin{proof}
 Let $D_0 \!= \textup{diam} (\Gamma \backslash X)$ and  suppose, by contradiction, that   there exists a  a sequence of lattices  $(\Gamma_j,X_j) $  converging to  $(\Gamma, X)$, with $M_j = \Gamma_j \backslash X_j$ not homeomorphic to $M$. 
Then, diam$(\Gamma_j \backslash X_j) \leq 2D_0 $ definitely, and by Proposition \ref{lemma-GH-compactness-packing}  the   $(\Gamma_j, X_j) $'s and their limit   $(\Gamma, X)$ belong  to $\textup{CAT}_0(P_0,r_0,2D_0)$, for some $P_0$ and $r_0$.
Since the limit group is discrete by assumption, we deduce that we are in the noncollapsing setting, by Theorem \ref{theo-collapsed}. 
Moreover, by Corollary \ref{prop-Riemannian-close-topological}, we know that  the  $X_j$'s are  topological manifolds for $j \gg 0$, 
thus we are in the nonsingular setting of convergence. Then, by Theorem \ref{theo-noncollapsed}, we conclude that $\Gamma_j$ is isomorphic to $\Gamma$ for $j\gg 0$,  in particular $\Gamma_j$ is torsion-free as well. 
 It follows that the  $M_j$'s are locally   CAT$(0)$-spaces which  converge in the GH-distance  to the Riemannian manifold $M$. 
 Since we are in the noncollapsing case,  the injectivity radius of all the the spaces $M_j$ is uniformily bounded away from zero: then,  applying \cite[Theorem 1.8]{Na02}  (with $\delta$ that can be taken arbitrarily small since $M$ is a Riemannian manifold), implies that $M_j$ is homeomorphic to $M$ for $j\gg 0$.
\end{proof}

\subsection{Limit dimension and limit  Euclidean factor}
\label{sub-characterization}${}$

\noindent If  $X \in \textup{CAT}_0(P_0,r_0,D_0)$  then dim$(X) \leq n_0=P_0/2$, by  Proposition \ref{prop-packing}. Moreover, if  $ (X_j,x_j) \underset{\textup{pGH}}{\longrightarrow} (X_\infty,x_\infty)$ then,  by \cite[Theorem 6.5]{CavS20}
$$\text{dim}(X_\infty) \leq \liminf_{j \rightarrow +\infty} \text{dim}(X_j).$$
The following theorem precisely relates the collapsing (as defined in \ref{defsetting}) in the standard setting of convergence  to the dimension of the limit quotients, and is a direct consequence of the convergence Theorems \ref{theo-noncollapsed} \&  \ref{theo-collapsed}.


\begin{theo}[Characterization of collapsing]
	\label{theo-collapsing-characterization}  ${}$\\
Let $(\Gamma_j,X_j) \in \textup{CAT}_0(D_0)$ be a sequence of lattices converging in the pointed equivariant \textup{GH}-distance to 
$(\Gamma_\infty,X_\infty)$, 
and let $M_j = \Gamma_j \backslash X_j$, $M_\infty = \Gamma_\infty \backslash X_\infty$.
Then:
	\begin{itemize}[leftmargin=6mm] 
		\item[(i)] the sequence is non-collapsing if and only if $\textup{TD}(M_\infty) = \lim_{j\to +\infty} \textup{TD}(M_j)$;
		\item[(ii)] the sequence is collapsing if and only if $\textup{TD}(M_\infty) < \lim_{j\to +\infty} \textup{TD}(M_j)$.
	\end{itemize}
	Moreover, in the above characterizations,  the topological dimension \textup{TD} can be replaced by the Hausdorff dimension \textup{HD}.
\end{theo}

\noindent We just  need the following additional fact in order to establish the persistence of the dimension.
\begin{lemma}
	\label{lemma-dimension}
	Let $X$ be a proper, geodesically complete, $\textup{CAT}(0)$-space and $\Gamma < \textup{Isom}(X)$ be closed, cocompact and totally disconnected. Then the quotient metric space $M = \Gamma \backslash X$ satisfies
	$$\textup{HD}(M) = \textup{HD}(X) = \textup{TD}(X) = \textup{TD}(M).$$
\end{lemma}

\begin{proof}
	The quotient map $p\colon X \to M$ is open because it is the quotient under a group action. Moreover the fiber $p^{-1}([y])$ is discrete for all $[y] \in M$.\linebreak 
	Indeed, suppose  there are points $g_jy$ accumulating to $x \in X$, with $g_j\in \Gamma$. \linebreak
	By Ascoli-Arzel\`a we can extract a subsequence,  which we denote  again $(g_j)$, converging to $g_\infty \in \Gamma$ with respect to the compact-open topology. Since $X$ is geodesically complete, then the stabilizer $\text{Stab}_\Gamma(y)$ is open (see \cite[Theorem 6.1]{CM09b}).  In particular,  $g_\infty^{-1}g_j y = y$ and   we have $g_jy = g_\infty y $  for $j \gg 0$.  
	This shows that the fibers are discrete. 
	Then, by \cite[Theorem 1.12.7]{Eng78}, we deduce that $\text{TD}(X) = \text{TD}( M)$. On the other hand, since the projection $p$ is $1$-Lipschitz,  we have $\textup{HD}(M) \leq \textup{HD}(X)$. Moreover,  $\textup{HD}(X) = \textup{TD}(X)$ because of \cite[Theorem 1.1]{LN19}. As $\textup{TD}(M) \leq \textup{HD}(M)$ always holds,  we get
	$$\textup{TD}(M) \leq \textup{HD}(M) \leq \textup{HD}(X) = \textup{TD}(X) = \textup{TD}(M). \qedhere$$
\end{proof}

\vspace{3mm}
\begin{proof}[Proof of Theorem \ref{theo-collapsing-characterization}]${}$\\
As the groups are all $D_0$-cocompact, by  \cite[Proposition 6.5]{CavS20} 
 we have 
	$$\text{TD}(X_\infty)=\text{HD}(X_\infty) = \lim_{j\to +\infty}\text{HD}(X_j) = \lim_{j\to +\infty}\text{TD}(X_j),$$
	so the limit exists. Moreover,   $\text{TD}(M_j) = \text{HD}(M_j) = \text{TD}(X_j) = \text{HD}(X_j)$	for every $j$  by Lemma \ref{lemma-dimension}.
	Therefore, we also have
	$$\text{TD}(X_\infty)=\text{HD}(X_\infty) = \lim_{j\to +\infty}\text{HD}(M_j)=\lim_{j\to +\infty}\text{TD}(M_j).$$
	Suppose first that the sequence is non-collapsing. In this case $M_\infty$ is isometric to the quotient of $X_\infty$ by a closed, cocompact, totally disconnected group by Theorem \ref{theo-noncollapsed}. Hence $\text{TD}(M_\infty) = \text{TD}(X_\infty) = \lim_{j\to +\infty}\text{TD}(M_j)$ by Lemma \ref{lemma-dimension},
	and similarly for the Hausdorff dimension. \\
	Suppose now  that the sequence is  collapsing.  Theorem \ref{theo-collapsed} says that $M_\infty$ is the quotient of a proper, geodesically complete, CAT$(0)$-space $X'_\infty$ of dimension strictly smaller than the dimension of $X_\infty$, by a closed, cocompact, totally disconnected group. Again Lemma \ref{lemma-dimension} then implies that $$\text{TD}(M_\infty) = \text{HD}(M_\infty) < \text{HD}(X_\infty) = \text{TD}(X_\infty) = \lim_{j\to +\infty}\text{TD}(M_j)$$
	and similarly for the Hausdorff dimension.\end{proof}
 
 Notice that since we proved that   $\lim_{j\to +\infty}\text{TD}(M_j)$ exists,  Theorem \ref{theo-collapsing-characterization} excludes to have  sequences $(\Gamma_j, X_j,x_j)$ converging   with mixed behaviour  (that is, such that along some subsequence the convergence is collapsed, and along other subsequences it is non-collapsed).
\vspace{3mm}

Another consequence of the analysis of convergence in the nonsingular setting, combined with the renormalization Theorem  \ref{theo:renormalization}, is the persistence of the Euclidean factor at the limit:

\begin{cor}[Euclidean factor of  limits]\label{cor-euclfactor}${}$\\
Let $(\Gamma_j,X_j)\subseteq \textup{CAT}_0(D_0)$  be a sequence of nonsingular lattices  converging to $(\Gamma_\infty,X_\infty)$, and let $k_j$  be the dimension of the Euclidean factor of $X_j$. \linebreak Then   $k_\infty := \lim_{j\rightarrow +\infty} k_j$  exists and equals the dimension of the Euclidean factor of $X_\infty$.
\end{cor}

\begin{proof}
By Proposition \ref{prop-GH-ultralimit} it is sufficient to prove that for any  ultrafilter $\omega$ the ultralimit $k_\omega=\omega$-$\lim k_j$ is the same and equals the dimension of the Euclidean factor of $X_\omega$.  By the renormalization Theorem \ref{theo:renormalization} and by Theorem \ref{prop-vol-sys-dias}, up to passing to an action of the groups $\Gamma_j$ on new spaces $X_j'$ we may assume that the sequence is non-collapsing. Notice that this does not change the  isometry type of the  limit space $X_\omega$ because the spaces $X_j'$ provided by the renormalization Theorem \ref{theo:renormalization} are isometric to $X_j$. Then, we can apply the convergence Theorem \ref{theo-noncollapsed} and deduce that the limit group $\Gamma_\omega$ is discrete (and cocompact). Now, by \cite[Theorem 2]{CM19}, for any CAT$(0)$-space $X$ possessing a discrete, cocompact group $\Gamma$, the Euclidean rank coincides with the maximum rank of a free abelian commensurated subgroup $A < \Gamma$. But this maximum rank is stable under ultralimits in the standard, nonsingular setting of convergence without collapsing, since by Theorem \ref{theo-noncollapsed} the groups $\Gamma_j$ are all isomorphic to $\Gamma_\omega$ for $j\gg 0$.
\end{proof}

In contrast, if the quotient spaces $M_j=\Gamma_j \backslash X_j$, $M_\infty=\Gamma_\infty \backslash X_\infty$ split (or virtually split) metrically maximal dimensional tori  $\mathbb{T}^{k_j}$ and $\mathbb{T}^{k_\infty}$, the dimension $k_\infty$ of the limit tori in the collapsing case is generally different to the limit of the $k_j$'s.

\subsection{Collapsing when the isometry groups are Lie}
\label{sub-collapsinglie}${}$


\noindent We prove here the convergence Theorem \ref{teor-intro-Lie-nonsingular} stated in the introduction, where we suppose to have a converging sequence of nonsingular CAT$(0)$-orbifolds $M_j= \Gamma \backslash X_j$ where all the groups  Isom$(X_j)$ are Lie. 
This includes, for instance,  all collapsing sequences of  nonpositively curved Riemannian manifolds or, more generally, of CAT$(0)$-homology manifolds.

\noindent For the proof, we need the following
\begin{lemma}
\label{lemma-ultralimit-index}
Let $(\Gamma_j, X_j,x_j)$ be a sequence of isometry groups of pointed spaces, and let $\omega$ be  a non-principal ultrafilter.
If $\check \Gamma_j $ are subgroups of $\Gamma_j$ of index $[\Gamma_j:\check \Gamma_j ]\leq I$ for $\omega$-a.e.$(j)$, then $[\Gamma_\omega:\check \Gamma_\omega ]\leq I$.
\end{lemma}


\begin{proof}
Suppose to find $g_{\omega,1},\ldots,g_{\omega,{I+1}}\in \Gamma_\omega$ such that $g_{\omega,m}\Gamma_\omega' \neq g_{\omega,n}\Gamma_\omega'$ for all different $m,n \in\lbrace 1,\ldots,I+1\rbrace$. We write $g_{\omega,m} = \omega$-$\lim g_{j,m}$. We can find different $m, n \in\lbrace 1,\ldots,I+1\rbrace$ such that $g_{j,m}^{-1}g_{j,n} \in \Gamma_j'$ for $\omega$-a.e.$(j)$. By definition this implies that $g_{\omega,m}^{-1}g_{\omega,n} \in \Gamma_\omega'$, a contradiction.
\end{proof}
\vspace{2mm}

\begin{proof}[Proof of Theorem \ref{teor-intro-Lie-nonsingular}]
${}$\\
Let  $\omega$ be any non-principal ultrafilter, and let $(\Gamma_\omega, X_\omega)$ be the ultralimit of the $(\Gamma_j,X_j)$'s. We know that   $\omega$-$\lim \textup{sys}^\diamond (\Gamma_j, X_j)=0$.
Therefore,  the first assertion about  $M_j$ follows from Theorem \ref{theo-group-splitting-Lie}: for $j \gg 0$, the space $X_j$ splits as $Y_j  \times {\mathbb R}^{k_j}$, where ${\mathbb R}^{k_j}$ is the Euclidean factor, 
and  there is a normal subgroup $ \check \Gamma_j  = \check \Gamma_{Y_j}  \times {\mathbb Z}^{k_j} \triangleleft \Gamma_j$ of   index $ [\Gamma_j : \check \Gamma_j ] \leq  I_0$ such that
$\check M_j  = \check \Gamma_j \backslash X_j 
= \left( \check \Gamma_{Y_j} \backslash Y_j \right)  \times  \left( {\mathbb Z}^{k_j} \backslash {\mathbb R}^{k_j} \right) 
= \check N_j  \times {\mathbb T}^{k_j}$ and $\text{sys}^\diamond (\check \Gamma_{Y_j}, Y_j) \geq \sigma_0^\ast$. \\
Let  $(\check \Gamma_\omega, X_\omega)$,  
$(\check\Gamma_{Y_\omega}, Y_\omega )$ 
and $(A_\omega, \mathbb{R}^{k_\omega})$ be the ultralimits of the groups $\check \Gamma_j$, $\check\Gamma_{Y_j}$ and $\mathbb{Z}^{k_j}$ acting respectively on
 $X_j$, $Y_j$ and $\mathbb{R}^{k_j}$.
By Lemma \ref{lemma-ultralimit-product} we have   corresponding splittings   $X_\omega =Y_\omega \times {\mathbb R}^{k_\omega}$ and  $\check  \Gamma_\omega = \check\Gamma_{Y_\omega} \times A_\omega$,  with  $k_\omega = \omega$-$\lim k_j$; moreover, by Corollary \ref{cor-euclfactor}, $k_\omega=k_\infty$ is the dimension of the Euclidean factor of $X_\omega$.\\
The groups $\check\Gamma_{Y_j}$ act nonsingularly without collapsing on $Y_j$ by Theorem \ref{theo-group-splitting-Lie}, which implies that  $\check\Gamma_{Y_\omega}$ acts discretely on $Y_\omega$, by Theorem \ref{theo-noncollapsed}(ii). 
Therefore, the $\check N_j$'s converge without collapsing to a  nonsingular  orbispace $\check N_\omega=\check \Gamma_{Y_\omega} \backslash Y_\omega$. This  proves (i).\\
Moreover,  since $[\Gamma_j : \check \Gamma_j ] \leq I_0$,  Lemma \ref{lemma-ultralimit-index} implies that 
 $[\Gamma_\omega: \check \Gamma_\omega ] \leq I_0$ too. \linebreak
Therefore the  subgroup $\check \Gamma_\omega$ is open in $\Gamma_\omega$ and  $\check\Gamma_\omega^\circ = \Gamma_\omega^\circ = \mathbb{R}^{\ell}$, where $\ell \geq 1$ is provided by Theorem \ref{theo-collapsed}.
On the other hand, $A_\omega$  acts cocompactly by translations on $\mathbb{R}^{k_\omega}$, so  Lemma \ref{lemma-translations}  implies that $A_\omega$ is isomorphic to the group $\mathbb{Z}^{k_\omega - \ell'} \times \mathbb{R}^{\ell'}$ acting on $\mathbb{R}^{k_\omega} \cong \mathbb{R}^{k_\omega - \ell'} \times \mathbb{R}^{\ell'}$, for some $\ell' \geq 0$.\linebreak
 As $\check\Gamma_{Y_\omega}$ acts discretely on $Y_\omega$, it follows that $\check \Gamma_\omega^\circ = A_\omega^\circ = \mathbb{R}^{\ell'}$, therefore $\ell = \ell'$. \\ 
Moreover, since the groups ${\mathbb Z}^{k_j}$ converge to $A_\omega=\mathbb{Z}^{k_\omega - \ell} \times \mathbb{R}^{\ell}$, and the factors act separately by translations on the factors $ \mathbb{R}^{k_\omega - \ell}$ and  $\mathbb{R}^{\ell}$ of $ \mathbb{R}^{k_\omega}$ ,  it is clear from Lemma \ref{lemma-ultralimit-quotient} that the tori ${\mathbb T}^{k_j}={\mathbb Z}^{k_j} \backslash {\mathbb R}^{k_j}$ converge to the limit torus 
$\mathbb{Z}^{k_\omega - \ell} \backslash \mathbb{R}^{k_\omega - \ell} = {\mathbb T}^{k_\omega -\ell}$. This proves (ii).\\
Now, the group $\check \Gamma_\omega':=\check \Gamma_\omega / \check \Gamma_\omega ^\circ=  \check\Gamma_{Y_\omega} \times  \mathbb{Z}^{k_\omega - \ell} $ has index $\leq I_0$ in $ \Gamma_\omega' =\Gamma_\omega /  \Gamma_\omega ^\circ $; \linebreak and, since the first one is discrete, it follows that also $ \Gamma_\omega' $ is discrete. \linebreak
By Theorem  \ref{theo-collapsed} it follows that the spaces $M_j$  converge to  $M_\omega = \Gamma_\omega' \backslash X_\omega'$, where $X_\omega' =\Gamma_\omega^\circ \backslash X_\omega$, hence $M_\omega$   is a nonsingular orbispace, proving (iii).\\
Finally, remark that since $\check \Gamma_j$ is normal in $\Gamma_j$ for all $j$, then  $\check \Gamma_\omega$ is normal in $\Gamma_\omega$ (cp. \cite[Lemma 6.14.(iii)]{CavS20bis}), and in turns this implies that $\check \Gamma_\omega' \triangleleft  \Gamma_\omega' $. 
 Therefore,  
this space $M_\omega $ is the quotient of $\check M_\omega := \check \Gamma_\omega' \backslash X_\omega'$ by the group $\Lambda_\omega = \Gamma_\omega' / \check \Gamma_\omega'$ which has cardinality $\leq I_0$. \\
Moreover $Y_\omega$ has no Euclidean factors because of Corollary \ref{cor-euclfactor}, so $\textup{sys}^\diamond(\check \Gamma_{Y_\omega}) \ge \sigma_0^\ast$ by Theorem \ref{theo-group-splitting-Lie}.
This proves (iv).
\end{proof}

As a result of  the convergence theorems proved in  this section, we obtain the compactness  Corollary \ref{theo-intro-compactness} 
for  CAT$(0)$-homology orbifolds.  
Recall the class $ {\mathcal H}{\mathcal O} \text{-CAT}_0 (P_0, r_0, D_0)$ 
   defined in the introduction,  consisting of all  the quotients $M=\Gamma \backslash X$  with 
$(\Gamma,X)$
   belonging to $ \text{CAT}_0 (P_0, r_0, D_0)$, where   $X$ is supposed, in addition, to be a homology manifold.

\begin{proof}[Proof of Corollary \ref{theo-intro-compactness}] 
${}$\\
Let $M_j=\Gamma_j \backslash X_j$ be a sequence in ${\mathcal H}{\mathcal O} \text{-CAT}_0 (P_0,r_0, D_0)$. Since the $X_j$ are   homology manifolds, then $\Gamma_j$ is nonsingular and Isom$(X_j)$ is a Lie group for every $j$, as recalled in Section \ref{subsec-systole} and by Proposition \ref{cor-homology-Lie}.\\
By \cite[Theorem 6.1]{CavS20},  the class $\text{CAT}_0 (P_0,r_0, D_0)$ is compact with respect to the Gromov-Hausdorff convergence. 
Moreover,  if   $(X_j,x_j) \underset{\textup{pGH}}{\longrightarrow} (X_\infty,x_\infty)$, then $X_\infty$ is again a   homology manifold, by  \cite[Lemma 3.3]{LN-finale-18}.\\
Now, by Theorem \ref{theo-collapsing-characterization},   the sequence 
$( \Gamma_j,X_j)$
is either non-collapsing, or  collapsing. 
In the first case,   by Theorem \ref{theo-noncollapsed}  the quotients $M_j$ converge to some space $M_\infty = \Gamma_\infty \backslash X_\infty$, for  a discrete subgroup  $ \Gamma_\infty$ of Isom$(X_\infty)$, hence it belongs to  ${\mathcal H}{\mathcal O} \text{-CAT}_0 (P_0,r_0, D_0)$.
In the second case,   the $M_j$'s converge   to a lower dimensional limit space $M_\infty$, which by Theorem \ref{teor-intro-Lie-nonsingular}  is isometric to  the quotient  of the space $X_\infty'$ (a proper, geodesically complete,  $(P_0,r_0)$-packed CAT$(0)$-space, by Theorem \ref{theo-collapsed}) 
by the discrete, $D_0$-cocompact group $\Gamma_\infty'$.
Moreover, if $X_\infty = Y_\infty \times {\mathbb R}^{k_\infty}$ is the splitting of the Euclidean factor,  we have $X_\infty' =   Y_\infty \times  {\mathbb R}^{k_\infty - \ell}$ by the proof of Theorem \ref{theo-collapsed} and by Corollary \ref{cor-euclfactor}. By the Kunneth's formula it then follows that $Y_\infty$, and in turn $X_\infty'$, are homology manifolds; 
therefore,  $M_\infty$ still belongs to ${\mathcal H}{\mathcal O} \text{-CAT}_0 (P_0,r_0, D_0)$.
\end{proof}

%

\subsection{Isolation of Euclidean spaces and entropy rigidity}
${}$

\noindent We  prove in this section the rigidity and pinching Corollaries \ref{cor-flats} \&  \ref{cor-entropy} stated in the introduction.   


\begin{proof}[Proof of Corollary \ref{cor-flats}]${}$\\
	If the thesis is false,  there exists a sequence of lattices   $(\Gamma_j, X_j) \in \textup{CAT}_0(D_0)$  such that the $X_j$'s converge   to $\mathbb{R}^n$ in the pointed GH-distance,  but with $X_j$  not isometric to $\mathbb{R}^n$ for every $j$. 
	By Propositions \ref{prop-GH-ultralimit}  and   \ref{lemma-GH-compactness-packing}, there exist   $P_0,r_0>0$ such that $(\Gamma_j,X_j) \in \textup{CAT}_0(P_0,r_0,D_0)$ for all $j$, and, up to extracting a subsequence, we may assume  that the sequence $(\Gamma_j, X_j)$ converges in the equivariant pointed GH-distance   to $(\Gamma, \mathbb{R}^n)$ for some closed  subgroup $\Gamma < \textup{Isom}(\mathbb{R}^n)$. 
	 By the renormalization procedure of Theorem \ref{theo:renormalization}, which does not change  the isometry type of the spaces $X_j$ and keeps the diameter of the quotients $\Gamma_j \backslash X_j$ bounded by $\Delta_0$, we may assume that the sequence is noncollapsing. Moreover, the spaces $X_j$ are topological manifold for $j\gg 0$, by Corollary \ref{prop-Riemannian-close-topological}. Therefore we are in the nonsingular, non-collapsed setting of convergence, and we deduce from Theorem \ref{theo-noncollapsed}(ii) that the limit $\Gamma$ is
discrete, and  $\Gamma_j$ is isomorphic to $\Gamma$ for $j\gg 0$. \linebreak
In particular $\Gamma_j$  is virtually abelian of rank $n$. Then, \cite[Corollary C]{AB97} 
says that $X_j$ is  isometric to $\mathbb{R}^n$, leading to the contradiction.
\end{proof}

We conclude by proving Corollary \ref{cor-entropy}. The idea is that, if we assume to have a sequence of uniformily packed and uniformily cocompact CAT$(0)$-spaces  with smaller and smaller entropy,   we can pass to the limit and, by a  a continuity argument, deduce that the entropy of the limit is zero;  then we conclude again from \cite[Corollary C]{AB97} \linebreak saying that a space with zero entropy has to be flat.\\
 However, the continuity argument is delicate: it is known in some cases, under the assumption that the group actions are torsion-free  (see \cite[Proposition 38]{Reviron08} and  \cite{Cav21ter}).  Here we will use  a continuity result (Proposition \ref{prop:continuity_entropy}) for the entropy of spaces admitting nonsingular lattices, possibily with torsion, which will be proved in the Appendix, bypassing the techniques of \cite{Reviron08} and working directly at the level of universal covers.

%

\begin{proof}[Proof of Corollary  \ref{cor-entropy}]${}$\\
Reasoning by contradiction, assume that there exists a sequence $(\Gamma_j,X_j) \in \textup{CAT}_0(P_0,r_0,D_0)$   such that Ent$(X_j) \to 0$ but $X_j$ is not flat for every $j$, and (up to a subsequence) let $(\Gamma_\infty, X_\infty)$ be the limit. 
After the renormalization of the $X_j$'s given by Theorem \ref{theo:renormalization}, which does not affect the packing constants and keeps the diameters  of  $\Gamma_j \backslash X_j$ uniformily bounded by a constant $\Delta_0=\Delta_0(P_0, D_0)$, we may assume that the sequence is noncollapsing. Moreover,  by definition,  the renormalization does not change  the entropy of $X_j$ either (as the renormalized spaces are isometric to $X_j$). \linebreak
We therefore are in the noncollapsing, nonsingular setting, so we deduce from  Theorem \ref{theo-noncollapsed} that the limit group $\Gamma_\infty$ is still discrete, nonsingular and isomorphic to $\Gamma_j$ for $j\gg 0$.  
Better, by Proposition \ref{prop-vol-sys-dias}   we can choose basepoints $x_j \in X_j$ such that  sys$(\Gamma_j, x_j)$ (not only the free-systoles) are all larger than some constant $s'_0 = s'_0(P_0,r_0,\Delta_0)>0$, and the same   lower bound on the systole passes to the limit at a limit basepoint $x_\infty \in X_\infty$. We are now in position to apply Proposition \ref{prop:continuity_entropy} to infer that Ent$(X_\infty) = \lim_{j \rightarrow +\infty} \text{Ent}(X_j) = 0$. The group $\Gamma_\infty$ is amenable, because it has subexponential growth, and so it is $\Gamma_j$ for $j\gg 0$. Therefore $X_j$, for $j\gg 0$, is isometric to some Euclidean space $\mathbb{R}^n$ by \cite[Corollary C]{AB97}, which gives the contradiction.
\end{proof}

\appendix

\vspace{5mm}	
\section{Continuity of  entropy under GH-approximations}
\label{appentropy}

\begin{prop}
	\label{prop:continuity_entropy}
	There exists $\varepsilon = \varepsilon(s_0,D_0,\delta)>0$ such that 
	if $\Gamma$ and $\Gamma'$ are $D_0$-lattices of proper, geodesic, simply connected spaces $X, X'$ respectively, \linebreak with 
	$\min\{ \textup{sys}(\Gamma,x),  \textup{sys}(\Gamma',x') \} \geq s_0$ and 
	$d_{\textup{eq-pGH}}((\Gamma,X,x),(\Gamma',X',x')) < \varepsilon$
then	
\begin{equation}
\label{eq:1pmdelta}  \left\vert\frac{\textup{Ent}(X)}{\textup{Ent}(X')} - 1 \right\vert \leq \delta.
\end{equation}
\end{prop}
 
	We divide the proof in several steps. We will first prove that the groups $\Gamma, \Gamma'$ have canonical generating sets which make them isomorphic as marked groups, and the marked isomorphism is given by the approximation morphisms. Then, we quantify the distortion of this isomorphism, showing that it is a $(1+\delta, C)$-quasi-isometry, which will imply \eqref{eq:1pmdelta}. 
	
\noindent The value of $\varepsilon$ appearing in the statement  can be taken equal to  $\frac{1}{3L}$, where	
$$L = \max \left\{ 8D_0, \frac{1}{s_0},  \frac{2}{\delta}  \left( D_0 + \sqrt{D_0^2 + \frac{\delta}{3} }  \,\right) \right\}.$$ 

\begin{proof}[Proof of Proposition \ref{prop:continuity_entropy}]${}$\\
Let $(f,\phi,\psi)$ be an equivariant $\varepsilon$-approximation between $(\Gamma,X,x)$ and $(\Gamma',X',x')$ for $\varepsilon=\frac{1}{3L}$, with $f(x)=x'$ according to  the Definitions   \ref{defin:equivariant_GH_approximation}\&\ref{defin:equivariant_GH_distance}.
Recall the subset $\overline \Sigma_r(\Gamma, X,x)$  defined in \eqref{defsigma}, and set, for short,
$$\overline \Sigma_{r}=\overline \Sigma_{r} (\Gamma, X,x), \hspace{1cm} \overline \Sigma_{r}'= \overline \Sigma_{r} (\Gamma', X',x') .$$ 
Notice that, $\overline \Sigma_{L}$ and $\overline \Sigma'_{L}$   are generating sets  for $\Gamma$ and $\Gamma'$ respectively, since $L>2D_0$ (see for instance \cite[Appendix to \S3]{Ser80}).

\begin{lemma}\label{lemma:bijection}
The map $\phi: \overline \Sigma_{3L} \rightarrow 	\overline \Sigma'_{3L}$ is a bijection with  inverse  $\psi$, and sends the identity element of $\Gamma$ in the identity of $\Gamma'$.
\end{lemma}	
%
 
 \begin{proof}
To show the surjectivity, let $g' \in \overline \Sigma'_{3L}$. Then, setting $g=\psi(g')$, we have that $g \in \overline \Sigma_{3L}$ ,  $ d(f(gx), g' x' ) < \varepsilon$ and  $ d(f(gx), \phi(\psi(g')) x' ) < \varepsilon$, by Definition \ref{defin:equivariant_GH_approximation}. Hence
  $$ d(\phi(g)x', g'x') < 2\varepsilon < s_0$$
  which implies  that  $\phi(g)=g'$ since  $\textup{sys}(\Gamma',x') \geq s_0$. 
  Assume now that $g_1,g_2 \in \overline \Sigma_{3L}$ satisfy $\phi(g_1)=\phi(g_2)=g' \in  \Sigma'_{3L}$. 
 Then, by Definition \ref{defin:equivariant_GH_approximation}
  $$d(g_i x, \psi(g')x) < d(f(g_ix),f(\psi(g')x)) + \varepsilon, \mbox{ for } i=1,2.$$
On the other hand, by Definition  \ref{defin:equivariant_GH_approximation}, 
$ d(f(g_i x), \phi(g_i)f(x))=   d(f(g_i x), g'x') < \varepsilon$ and 
$ d(f(\psi(g') x), g'x') < \varepsilon$,   therefore we deduce that 
  $$d(g_i x, \psi(g')x) \le   d(f(g_i x), g'x')  + d(f(\psi(g') x), g'x')  +   \varepsilon < 3\varepsilon \le s_0,$$
 which implies   that $g_1= \psi(g')=g_2$ as  $\textup{sys}(\Gamma,x) \geq s_0$. So, $\phi$ is a bijection with inverse $\psi$.  
 Finally $d(f(e \cdot x),\phi(e)  f(x))= d(x',\phi(e)   x') < \varepsilon <s_0$, therefore $\phi(e)=e$ necessarily.
  \end{proof}

\begin{lemma}\label{lemma:isomorphism}
The maps $\phi,\psi$ induce isomorphisms $\tilde \phi: \Gamma \rightarrow \Gamma'$, $\tilde \psi: \Gamma' \rightarrow \Gamma$ inverse to each other (which coincide with $\phi$ and $\psi$ on  $\overline \Sigma_{L}$ and $ \overline \Sigma'_{L} $).
\end{lemma}

 \begin{proof}
 Let us first show that the map $\phi \colon \overline \Sigma_L \to \Gamma'$ satisfies 
 \begin{equation}\label{eq:hom}
 \phi(s_1s_2s_3) = \phi(s_1)\phi(s_2)\phi(s_3),  \mbox{ for every } s_1,s_2,s_3 \in \overline \Sigma_L 
 \end{equation}
and that the same holds for the map $\psi \colon \overline \Sigma'_{L} \to \Gamma$.
Indeed,  applying several times the Definition \ref{defin:equivariant_GH_approximation}  we find
\vspace{-5mm}

	$$ d(\phi(s_2s_3)f(x), f(s_2s_3x)) < \varepsilon,  $$ 
	$$ d(f(s_2s_3 x), \phi(s_2)f(s_3x)) < \varepsilon,  $$
	$$ d(\phi(s_2)f(s_3 x), \phi(s_2)\phi(s_3)f(x)) = d(f(s_3 x), \phi(s_3)f(x)) < \varepsilon$$
and by the triangular inequality we get  $d(\phi(s_2s_3)x', \phi(s_2)\phi(s_3)x') < 3\varepsilon \le s_0$.
 Therefore   $\phi(s_2s_3) = \phi(s_2)\phi(s_3)$, as  $\textup{sys}(\Gamma',x') \geq s_0$.  Arguing in a similar way we also deduce that $\phi(s_1s_2s_3) = \phi(s_1)\phi(s_2s_3) = \phi(s_1)\phi(s_2)\phi(s_3)$. \\
 The proof for the map $\psi$   is identical.\\
Let now   $\tilde{\phi}\colon {\mathbb F}( \overline  \Sigma_{L}) \to \Gamma'$ be the homomorphism from the free group with basis $ \overline  \Sigma_{L}$ to $\Gamma'$ defined by extension of $s_i \mapsto \phi(s_i)$ for every $s_i \in \overline  \Sigma_{L}$.  
Recall that by \cite[App., Ch.3]{Ser80}, the group	 $\Gamma$ admits a presentation $\Gamma = \langle  \overline  \Sigma_L \vert \mathcal{R} \rangle$
(that is, $\Gamma$ is isomorphic to the quotient of the free group ${\mathbb F}(\overline \Sigma_{L})$ by the normal closure of $\mathcal{R}$),  where every relator in $\mathcal{R}$ has length at most $3$ in the alphabet $ \overline  \Sigma_L$. \linebreak
	Then, if $w =s_1s_2s_3 \in \mathcal{R}$, with each $s_i \in \overline    \Sigma_L$, we deduce by \eqref{eq:hom} that
$$\tilde{\phi}(g) = \tilde{\phi}(s_1s_2s_3) = \phi(s_1)\phi(s_2)\phi(s_3) = \phi(s_1s_2s_3) = \phi(e)  $$
which is the identity in $\Gamma'$  by Lemma \ref{lemma:bijection}. This shows that the map $\tilde \phi$ induces a homomorphism  $\Gamma \to \Gamma'$, still denoted by $\tilde \phi$, which coincides with $\phi$ on $\overline{\Sigma}_L$.  
Similarly,   $\psi$ induces a  homomorphism $\tilde{\psi}\colon \Gamma' \to \Gamma$ such that $\tilde \psi=\psi$ on $\overline{\Sigma}'_L$.\\
To conclude that $\tilde \phi$ is an isomorphism with inverse $\tilde \psi$ is enough to remark that they coincide with $\phi$ and $\psi$ on the generating sets $\overline \Sigma_L$,  $\overline \Sigma'_L$ respectively, and that $\phi$ and $\psi$ are inverse to each other on these sets, by Lemma  \ref{lemma:bijection}. 
%
  \end{proof}

\begin{lemma}\label{lemma:quasi-isometry} 
If $d(x,gx) \geq L$ and  $d(x',g'x') \geq L$ we have:
$$d(x',\tilde{\phi}(g)x') \le d(x,gx)(1+\delta) + 2(D_0 + \varepsilon)$$
$$d(x,\tilde{\psi}(g)x) \le d(x',g'x')(1+\delta) + 2(D_0 + \varepsilon)$$
\end{lemma}	
	
 
 \begin{proof}	
 Let $d=d(x,gx) \ge L$, let $c = [x,gx]$ be a geodesic between $x$ and $gx$, and set $n := \lceil \frac{2d}{L} \rceil$. 
 Let us define $x_i := c(i\frac{L}{2})$ for $i=0,\ldots,n-1$ and $x_n := gx$. For every $i = 0,\ldots,n$ we choose $g_i\in \Gamma$ such that $d(g_ix, x_i) \le 2D_0$, with $g_0 = \text{id}$ and $g_n = g$. We define $h_i = g_{i-1}^{-1}g_i$ for  $i = 1,\ldots,n$ and we observe that $g = h_1\cdots h_n$. Notice that each $h_i$ belongs to $\overline \Sigma_L$ as 
	$$d(x,h_ix) \le \frac{L}{2} + 4D_0 \le L.$$
 By definition of $\tilde \phi$ we then have $\tilde{\phi}(g) =  \phi (h_1)\cdots  \phi (h_n)$, and the usual application of  Definition \ref{defin:equivariant_GH_approximation} yields
$$d(x', \phi(h_i)x') \le d(f(x), f(h_i x)) + \varepsilon \leq d(x,h_i x) + 2\varepsilon$$
therefore
	\begin{equation*}
		\begin{aligned}
			d(x',\tilde{\phi}(g)x') &\le d(x',\phi(h_1) x') +\ldots +d(x',\phi(h_n) x')  \\
			&\le d(x,h_1x) +\ldots + d(x,h_n x) + 2n\varepsilon \\
			&\le d(x,x_1) + \ldots + d(x_{n-1}, x_n) + 2n(D_0 + \varepsilon) \\
			&= d + 2n(D_0+\varepsilon) = d \left(1 + \frac{2n}{d}(D_0 +\varepsilon) \right)  \\
			&\le d(x,gx)\left(1+\frac{4}{L}(D_0 + \varepsilon) \right) + 2(D_0 +  \varepsilon) \\
			&\le d(x,gx)(1+\delta) + 2(D_0+\varepsilon).
		\end{aligned}
	\end{equation*}
since $n \leq \frac{2d}{L}+1$, 
and $\frac{4}{L}(D_0 + \varepsilon)  \le \delta$ by our choice of $L$. \\
The proof for $\tilde{\psi}$ is the same.
 \end{proof}

\vspace{3mm}
\noindent {\em End of proof of Proposition   \ref{prop:continuity_entropy}.}\\
Set $B^\ast(x,R)=B(x,R)- \overline B(x,L)$, and  observe that we can equivalently compute the entropy of $X$ as
\begin{equation*}
  \text{Ent}(X)  = \lim_{R\to +\infty}\frac{1}{R}  \log  \# \left( \Gamma x \cap  B^\ast(x,R) \right).
\end{equation*}
By Lemma \ref{lemma:quasi-isometry}, for every element $g\in \Gamma$ such that $gx \in B^\ast(x,R)$  we have $d(x',\tilde{\phi}(g)x') \le (1+\delta)R + 2(D_0+\varepsilon)$. Observe moreover that the points $\tilde{\phi}(g)x'$ are different for different $g$'s in $\Gamma$. Therefore
	$$\# \left( \Gamma x \cap  B^\ast(x,R) \right) \le \# \left( \Gamma' x' \cap  B(x',(1+\delta)R + 2(D_0+\varepsilon) \right)$$
	which yields
\begin{equation*}
\text{Ent}(X)  \le  (1+\delta) \cdot  \lim_{R\to +\infty} \frac{1}{R} \log  \# \left( \Gamma' x' \cap B(x',R) \right) = (1+\delta) \cdot \text{Ent}(X').
\end{equation*}
	Reversing the roles of $X$ and $X'$ we deduce that 
	$\text{Ent}(X') \le (1+\delta)\text{Ent}(X)$,  which proves \eqref{eq:1pmdelta} as $\frac{1}{1+\delta} > 1- \delta$.
\end{proof}

\bibliographystyle{alpha}
\bibliography{collapsing}

\end{document}